\documentclass[11pt,a4paper,reqno]{amsart}

\usepackage[T1]{fontenc}
\usepackage{amsmath,amssymb,amsfonts,amsthm,amstext}
\usepackage{mathrsfs}
\usepackage[mathscr]{eucal}
\usepackage{graphicx,epstopdf}
\usepackage{color,xcolor}
\usepackage{array}
\usepackage{enumitem}
\usepackage{setspace}
\usepackage[title]{appendix}
\usepackage{amscd,psfrag}
\usepackage{yhmath}
\usepackage{comment}
\usepackage{slashed}
\usepackage[normalem]{ulem}
\usepackage{indentfirst}
\usepackage{chngcntr}
\usepackage[margin=2.5cm]{geometry}
\usepackage{cite}
\usepackage[colorlinks,citecolor=blue,urlcolor=blue]{hyperref}

\makeatletter
\@namedef{subjclassname@2020}{%
  \textup{2020} Mathematics Subject Classification}
\makeatother

\allowdisplaybreaks[4]
\counterwithin{figure}{section}

\theoremstyle{plain}
\numberwithin{equation}{section}

\newtheorem{definition}{Definition}[section]
\newtheorem{theorem}[definition]{Theorem}
\newtheorem*{theorem*}{Theorem}

\newtheorem{lemma}[definition]{Lemma}
\newtheorem{corollary}[definition]{Corollary}
\newtheorem*{corollary*}{Corollary}
\newtheorem{proposition}[definition]{Proposition}

\newtheorem*{claim*}{Claim}
\newtheorem*{q*}{Question}

\theoremstyle{definition}
\newtheorem{remark}[definition]{Remark}
\newtheorem*{remark*}{Remark}
\newtheorem*{sideremark*}{Side Remark}

\newtheorem{example}[definition]{Example}
\newtheorem{convention}[definition]{Convention}

\newcommand{\blue}[1]{\textcolor{blue}{#1}}

\newcommand{\R}{\mathbb{R}}

\newcommand{\na}{\nabla}

\newcommand{\id}{{\rm Id}}

\newcommand{\p}{\partial}

\newcommand{\C}{\mathbb{C}}
\newcommand{\dd}{{\rm d}}

\newcommand{\two}{{\rm II}} 
\newcommand{\bra}{\left\langle}
\newcommand{\ket}{\right\rangle}

\newcommand{\err}{{\mathscr{E}}}

\newcommand{\mmp}{{\mom p}}
\newcommand{\1}{{\mathbf{1}}}

\newcommand{\ir}{{\mathcal{I}\mathcal{R}}}

\newcommand{\tens}{{T\left(\left(\R^d\right)\right)}}

\newcommand{\mom}{{\frac{\scrm}{m}}}

\newcommand{\fs}{{\mathfrak{S}}}

\newcommand{\sym}{{\bf sym}}

\newcommand{\scrm}{{\mathscr{M}}}
\newcommand{\scrd}{{\mathscr{D}}}
\newcommand{\scrc}{{\mathscr{C}}}
\newcommand{\scrg}{{\mathscr{G}}}
\newcommand{\fraka}{{\mathfrak{A}}}
\newcommand{\frakb}{{\mathfrak{B}}}
\newcommand{\one}{{\rm I}}

\newcommand{\ssig}{\widetilde{{\bf \Sigma}}}

\newcommand{\mm}{{(m)}}

\def\Xint#1{\mathchoice
{\XXint\displaystyle\textstyle{#1}}%
{\XXint\textstyle\scriptstyle{#1}}%
{\XXint\scriptstyle\scriptscriptstyle{#1}}%
{\XXint\scriptscriptstyle\scriptscriptstyle{#1}}%
\!\int}
\def\XXint#1#2#3{{\setbox0=\hbox{$#1{#2#3}{\int}$ }
\vcenter{\hbox{$#2#3$ }}\kern-.6\wd0}}

\def\dashint{\Xint-}

\DeclareMathOperator{\E}{\mathbb{E}}
\DeclareMathOperator{\Var}{Var}

\title{Small mass limit of expected signature for physical Brownian motion}
\thanks{To appear in \emph{Annales de l'Institut Henri Poincar\'e -- Probabilit\'es et Statistiques}.}

\author{Siran Li}
\address{S.~Li: School of Mathematical Sciences and CMA-Shanghai, Shanghai Jiao Tong University, No.~6 Natural Sciences Building,
800 Dongchuan Road, Minhang, Shanghai, China (200240)}
\email{siran.li@sjtu.edu.cn}

\author{Hao Ni}
\address{H.~Ni: Department of Mathematics, University College London, Room 603, 25 Gordon St, London WC1H 0AY, UK; and the Alan Turing Institute, 96 Euston Rd, London NW1 2DB, UK}
\email{h.ni@ucl.ac.uk}

\author{Qianyu Zhu}
\address{Q.~Zhu: Center for Computational Science and Engineering, Massachusetts Institute of Technology, 77 Massachusetts Ave, Cambridge, MA 02139}
\email{qianyu\_z@mit.edu}


\date{March 2026}

\begin{document}

\begin{abstract}
Physical Brownian motion describes the dynamics of a Brownian particle experiencing frictional force. It was investigated in the classical work [L.~S.\ Ornstein and G.~E.\ Uhlenbeck, {\em Phys.\ Rev.}\ {\bf 36} (1930)] as a physically meaningful approach to realising the standard ``mathematical'' Brownian motion, via sending the mass $m \to 0^+$ and performing natural scaling. The analysis was extended to a Brownian particle in an external magnetic field in [P.\ Friz, P.\ Gassiat, and T.\ Lyons, {\em Trans.\ Amer.\ Math.\ Soc.}\ {\bf 367} (2015)], discovering the new phenomenon that the area process associated to the
physical process converges --- but not to L\'{e}vy's stochastic area. In this paper, we carry out the singular limit analysis of a generalised stochastic differential equation (SDE) model encompassing the physical Brownian motion as a special case. We show that the expected signature of the solution $\left\{P_t\right\}_{t \geq 0}$ for the generalised SDE converges to a nontrivial tensor as $m \to 0^+$, at each degree in the tensor algebra and on each time interval $[0,T]$, through a delicate convergence analysis based on the graded PDE system for the expected signature of It\^{o} diffusion processes. Moreover, explicit solutions exhibiting intriguing combinatorial patterns are obtained when the coefficient matrix $\scrm$ in our SDE is diagonalisable. In the case of physical Brownian motion, $\left\{P_t\right\}_{t \geq 0}$ corresponds to the momentum of the particle (viewed as a rough path), and $\scrm$ is the stress tensor. Our work appears among the very first endeavours to study the singular limit of expected signature of diffusion processes, especially for nonzero initial datum $p=P_0$.
\end{abstract}

\maketitle

\tableofcontents

\smallskip

\section{Introduction}
This paper is devoted to the singular limit analysis of a general family of  stochastic differential equations (SDE) models, motivated by and encompassing the \emph{physical Brownian motion} that models the behaviour of objects subject to random impulses. The classical model of physical Brownian motion, first studied in the seminal work \cite{ou} by Ornstein--Uhlenbeck (1930), describes the dynamics of a Brownian particle experiencing friction. Apart from its  significance in physics, physical Brownian motion has also brought about considerable mathematical interests: a well-known approach to realise the standard Wiener process is to take its \emph{small mass limit} (\emph{a.k.a.} zero-mass limit or massless limit). One may scale the SDE for physical Brownian motion by the mass $m$ of the particle, and then send $m \to 0^+$, to obtain the almost sure uniform convergence in distribution to the ``mathematical'' Brownian motion.

\subsection{Physical Brownian motion} 
Let us discuss the physical Brownian motion, starting from dimension three and then generalising to arbitrary dimensions. The SDE studied in this work, Eq.~\eqref{main SDE, equation for P}, is more general than the classical physical Brownian motion. The set up of the physical Brownian motion is taken from Friz--Gassiat--Lyons \cite{friz2015physical}, which extends that in \cite{ou} by additionally including an external magnetic force.

Consider a particle in $\R^3$ with mass $m$, electric charge $q$, and position $x = x(t)$. Assume that it is subject to white noise $\xi$ in time (where $\xi=\xi(t)$ is the derivative of Brownian motion $\mathcal{B}$), and that it moves in a magnetic field $H$ subject to friction. Hence, the particle experiences friction $F_{\rm friction} = -A\dot{x}$ ($A \in {\rm Sym}^{3 \times 3}_+$ is a symmetric, strictly positive $3 \times 3$ matrix) and the Lorentz force $
F_{\rm Lorentz} = q\dot{x} \wedge H$. Since the cross product in dimension 3 can be represented by left-multiplication of anti-symmetric matrices, we may write $F_{\rm Lorentz} = qB\dot{x}$ for some $B \in \mathfrak{so}(3;\R)$.\footnote{All of our notations for matrix groups/algebras are standard; in particular, we write $\mathfrak{gl}(d;\R)$ for the space of $d \times d$ real matrices and similarly for $\mathfrak{gl}(d;\C)$.} Thus, by Newton's second law, $m\ddot x = -\scrm\dot x+\xi$ with $\scrm := A-qB$.

In this work, we focus on a generalised model in arbitrary dimension $d$:
\begin{align}\label{physical BM, x eqn}
&m\ddot x = -\scrm \dot{x} + \xi\quad \text{where $m>0$ and the matrix } \scrm \in \mathfrak{gl}(d;\R)\nonumber \\
&\qquad\qquad\qquad\qquad\text{ has eigenvalues with strictly positive real parts}.
\end{align}
\emph{Formally}, setting $m=0$ reduces the above equation to $\scrm \dot{x} = \xi = \dot{\mathcal{B}}$. For time-independent $\scrm$ it reduces to $\scrm x = \mathcal{B}$, which describes the standard ``mathematical'' Brownian motion. In addition, in terms of the momentum $P(t) = m\dot x(t),$ Eq.~\eqref{physical BM, x eqn} is equivalent to 
\begin{equation}\label{p eqn}
\dot{P} = -\scrm \dot{x} +\xi = -\frac{\scrm}{m}P + \dot{\mathcal{B}}.
\end{equation}
This paper is primarily concerned with the small mass limit (\emph{i.e.}, $m \to 0^+$) for Eq.~\eqref{physical BM, x eqn} or \eqref{p eqn}.

\subsection{From physical to mathematical Brownian motions: non-convergence of signature}
In the pioneering work~\cite{friz2015physical}, Friz--Gassiat--Lyons ascertained that the convergence of the physical Brownian motion to the standard Wiener process as $m \to 0^+$ is \emph{``less robust than it would appear''} --- Although the convergence holds almost surely in the uniform sense in distributions, generic nonlinear functionals $\mathcal{F}$  on the trajectories of the physical Brownian motion do not necessarily converge. That is, $\scrm x \to \mathcal{B}$ but $\mathcal{F}[\scrm x] \nrightarrow \mathcal{F}[\mathcal{B}].$

The primary example of the functional $\mathcal{F}$ analysed in \cite{friz2015physical} is the \emph{level-2 signature}. Stated with respect to the momentum $P_{[s,t]}$ for $0<s<t<\infty$ (see \cite[Proposition~1]{friz2015physical}), one considers 
\begin{align}\label{dp otimes dp}
\mathcal{F}\left[P_{[s,t]}\right] := \int_s^t\int_s^r \dd P_{r'} \otimes \dd P_r.
\end{align}
This is motivated by viewing the momentum as a \emph{control}, or equivalently, viewing the Brownian particle as a \emph{rough path}. It is proved in \cite[Theorem~1]{friz2015physical} that the quantity $\scrm x$ of the physical Brownian motion converges in a nontrivial way: Its area process 
deviates from that of the standard Brownian motion (\emph{i.e.}, L\'{e}vy's stochastic area). The relevant correction term has been identified with the small mass limit of the quantity in Eq.~\eqref{dp otimes dp}. See Friz--Hairer~\cite{FH} for a survey.

\subsection{Expected signature}

We continue the aforementioned investigations by studying the small mass limit of the \emph{expected signature} of momentum (as a rough path) at each level $n$. A key novelty of our work is that the initial momentum $p := P\big|_{t=0}$ is allowed to be arbitrary in $\R^d$.

 The signature of a path is a core object in rough path theory. It is a group homomorphism from the path space into the tensor algebra. The level-$n$ signature may be regarded as the non-commutative analogue of the degree-$n$ monomial in the polynomial ring. The \emph{expected signature}, \emph{i.e.}, the expectation of the signature,  is of considerable theoretical and practical significance. It serves as the moment generating function of a path-valued random variable when the path is random. The expected signature of Gaussian processes (Boehihardjo--Papavasiliou--Qian \cite{boedihardjo2013expected}; Cass--Ferrucci 
 \cite{cass2024wiener}), L\'evy processes (Friz--Shekhar 
 \cite{friz2017general}), and Schramm-Loewner Evolutions (Boehihardjo--Ni--Qian \cite{boedihardjo2014uniqueness}), etc., has been studied extensively. 
 
 In terms of theoretical aspects, the expected signature determines the law of the stochastic process if it has an infinite radius of convergence. See Chevyrev--Lyons \cite{chevyrev2016characteristic} and the recent improvement in Li--Lyu--Ni--Tao \cite{llnt}. Examples of such processes include the standard ``mathematical'' Brownian motion and the fractional Brownian motions with Hurst parameter in $\left]\frac{1}{4}, 1\right]$  up to a fixed time (Cass and Ferrucci \cite{cass2024wiener}). 
The expected signature also provides a rich statistical summary of underlying stochastic processes; \emph{e.g.}, the characteristic function of the L\'{e}vy area process of Brownian motion can be recovered from the expected signature (Levin--Wildon \cite{levin2008combinatorial}). 


On the other hand, the study of expected signature has led to advances in numerical methods for PDEs. Computations for the expected signature of Brownian motion give rise to the notion of ``cubature on Wiener space'', now an effective numerical tool for solving parabolic PDEs (Lyons--Victoir \cite{LV_2004}). Recently, applications of expected signature in machine learning, especially generative models for synthetic time series generation, have attracted considerable attention. Chevyrev--Oberhauser \cite{Chevyrev_2022} introduced the ``Sig-MMD metric'' based on the expected signature and demonstrated its usefulness in two-sample hypothesis testing. Independently, \cite{ni2020conditional,ni2021sig} proposed the Sig-$W_1$ metric, which aligns with Sig-MMD but is instead motivated by the Wasserstein-1 distance on the signature space, and established an effective framework for learning generative models using the Sig-$W_1$ metric as a discriminator in time series generation. In comparison to popular Wasserstein generative adversarial networks (WGAN), Sig-WGAN effectively reduces the min-max game of WGAN to a supervised learning problem, thus leading to faster, more stable training and improved empirical performance on time series data.

An important question in both theory and practice is to determine the expected signature of stochastic processes. Lyons--Ni \cite{ni2012expected} proposed a PDE approach to the expected signature: an infinite, recursive PDE system characterising the expected signature $\Phi$ of time-homogeneous diffusion processes has been derived. As a mapping $\Phi:[0,T] \times \R^d\to\tens$, the tensor algebra, and for each $n \in \mathbb{N}_{\geq 2}$ its $n^{\text{th}}$-level truncation $\Phi_n$ satisfies a parabolic PDE system in terms of $\Phi_{n-1}$ and $\Phi_{n-2}$. See \S\ref{sec: prelim} below for details. This PDE approach has been applied to the analysis of the expected signature of Brownian motion up to the first exit time (\emph{cf.} \cite{lyons2015expected, ours, bdmn}). 

\subsection{Main result and key novelty}

Our main objective is to identify and rigorously prove the limit as $m \to 0^+$ of the expected signature of the solution for the SDE:
\begin{equation}\label{main SDE, equation for P}
\dd P_t = -\mom P_t \,\dd t +\dd W_t,\qquad 
P\big|_{t=0} = p.
\end{equation}
The expected signature $\Phi: [0,T] \times \R^d \to \tens$ 
 of $P_t$ is defined as
\begin{equation}\label{exp sig, def}
\Phi^\mm(t,p) :=\E^p\left[S\left(P_{[0, t]}\right)\right]=\E\left[S\left(P_{[0, t]}\right)\big| P_0=p\right],
\end{equation}
where the signature $S\left(P_{[0, t]}\right):=\left(1, S^1(P_{[0,t]}), S^2(P_{[0,t]}), \cdots\right)$ is defined in the Stratonovich sense:
\begin{equation*}
S^n(P_{[0,t]}):=\idotsint\limits_{0\leq t_1<\dots<t_n\leq t}\dd P_{t_1}\otimes\cdots\otimes\dd P_{t_n}.
\end{equation*}
The superscript $^\mm$ in $\Phi^\mm$ signifies the dependence on the small parameter $m>0$.\footnote{
As discussed before, Eq.~\eqref{main SDE, equation for P} is primarily motivated by and encompasses the model of physical Brownian motion, but it represents a more general class of SDE than the original physical model. Throughout we  still refer to Eq.~\eqref{main SDE, equation for P} as the equation for ``physical Brownian motion'', and refer to $\left\{P_t\right\}_{t\geq 0}$ and $m\to 0^+$ as the ``momentum path'' and the ``small mass limit''.}


Friz--Gassiat--Lyons proved  in the pioneering work \cite{friz2015physical} that the strong solution $P_t$ for Eq.~\eqref{main SDE, equation for P}, identified via its canonical lift as a rough path (termed the ``momentum path'' in the sequel), converges non-trivially in the small mass limit. Indeed, the first level truncation of its signature tends to zero, while the second level  on $[0,t]$ converges to $\left(\scrm \scrc - \frac{\id}{2}\right)t   \in  \mathfrak{gl}(d;\R)$ for each $0<t<\infty$, where $\scrc$ is a matrix determined by $\scrm$. It coincides with the covariance matrix of the Gaussian law of the limiting distribution of the scaled variable $P_t \slash \sqrt{m}$. A convergence theorem on the small mass limit of the displacement $X$ (from $\scrm X$ to $\widehat{\mathbb{W}}$) is also proved, where $\widehat{\mathbb{W}}$ is the rough path of the standard Wiener process \emph{with second level perturbed by $\scrm \scrc - \frac{\id}{2}$}.


Our main theorem of this paper is as follows.\footnote{For simplicity, all matrix and tensor norms are taken to be the $\ell^\infty$-norm throughout this paper. See Convention~\ref{convention} for details.} Here and hereafter, we denote the averaged integral of a function $f: I \subset \R\to\R$ as
\begin{align*}
    {\dashint}_I f := \frac{1}{|I|}\int_I f(t)\,\dd t.
\end{align*}


\begin{theorem}[Main Theorem]\label{thm: main, Dec24}
 Let $\Phi_{n}^\mm: \R^d\times\R_+ \to \left(\R^d\right)^{\otimes n}$ be the degree-$n$ expected signature of the momentum path $\left\{P_t\right\}$. Suppose $\|\scrm\| \leq \Lambda$ and that there is  $K>0$ such that for each $\delta>0$,  $\max\left\{\left\|e^{-\scrm^*\delta}\right\|, \left\|e^{-\scrm\delta}\right\|\right\} \leq Ke^{-\lambda\delta}.$ Here $\lambda \leq\Lambda$ and $K$ are finite  constants depending only on $\scrm$. Then for each $p\in\R^d$, $t>0$, and $n = 1,2,3,\ldots$, we have that
\begin{align*}
\Phi_n^\mm(p,t) &= \sum_{k=0}^{\lfloor {n}\slash{2}\rfloor} \left\{\fraka_{n-2k}^\mm(p,t) \otimes \left[ \left(\scrm \cdot\dashint_0^{{t}\slash{m}}\ssig_\sigma \,\dd \sigma -\frac{\id}{2}\right)t\right]^{\otimes k}  \right\}  + {\rm Error}_n^\mm(p,t),
\end{align*}
where ${\rm Error}_n^\mm(p,t)$ is a polynomial  in $p$ of degree no more than $\max\{n-2,0\}$, and 
\begin{align*}
\fraka_n^\mm(p,t) = \idotsint\limits_{0\leq t_1<\dots<t_n\leq t} \,\,\bigotimes_{j=1}^n  \left( -\mom e^{-\mom t_j} p\right)\, \dd t_1\ldots\,\dd t_n;\qquad\ssig_\sigma := \int_0^\sigma e^{-\scrm \varsigma} e^{-\scrm^*\varsigma}\,\dd\varsigma.
  \end{align*}
The error term satisfies
\begin{align*}
\left\|{\rm Error}_n^\mm(p,t)\right\| \leq C\left(K,\Lambda, \lambda^{-1}, t, d, n\right)\, \left(1+|p|^{\max\{n-2,0\}}\right)m.
\end{align*}
Thus, it vanishes  as $m\to 0^+$ for each fixed $n,d \in \mathbb{N}$, $t \in ]0,\infty[$, $p \in \R^d$, and matrix $\scrm$.  

The small mass limit can be explicitly expressed as follows: for each $q \in \mathbb{N}$, denote
\begin{align*}
    \overline{\fraka}_{q}(p) := \idotsint\limits_{0\leq t_1<\dots<t_q < \infty}\left(-\scrm e^{-\scrm t_1}p\right)\otimes\cdots\otimes\left(-\scrm e^{-\scrm t_q}p\right)\,\dd t_1\ldots\,\dd t_q,\quad \scrc := \ssig_\infty = \int_0^\infty e^{-\scrm t}e^{-\scrm^* t}\,\dd t.
\end{align*}
Then
\begin{align*}
\lim_{m \to 0^+} \Phi^\mm_n(t,p) &= \sum_{k=0}^{\lfloor {n}\slash{2}\rfloor} \left\{\overline{\fraka}_{n-2k}(p) \otimes  \left(\scrm \scrc -\frac{\id}{2}\right)^{\otimes k}  \right\}t^k = \sum_{k=0}^{\lfloor {n}\slash{2}\rfloor} \left\{\overline{\fraka}_{n-2k}(p) \otimes  \left(\frac{\scrm \scrc - \scrc\scrm^*}{2}\right)^{\otimes k}  \right\}t^k
\end{align*}
in the $\ell^\infty$-norm of tensors, with the convention $\overline{\fraka}_0^\mm(t,p) \equiv \overline{\fraka}_0 = \frac{1}{n!} \1 = \left(\frac{1}{n!},0,0,\ldots\right).$ In particular, the massless limit at level $n$ is a polynomial of degree $n$ in $p$ and of degree $\lfloor {n}\slash{2}\rfloor$ in $t$. 
\end{theorem}

When $\scrm$ is diagonalisable, an explicit formula for the limit is given in \S\ref{sec: diag}, Example~(Diagonalisable $\scrm$).

\begin{remark}
The expected signature up to level $2$ is studied in \cite{friz2015physical}. The initial condition $p=0$ is essential for the proof therein, which uses the ergodic theorem and hands-on computations via the properties of Gaussian processes. In order to extend the results to arbitrary initial datum $p \in \R^d$ and expected signature at arbitrary degree, we employ distinctively different approaches based on Eq.~\eqref{exp sig, def},  the recursive PDE system proposed in \cite{ni2012expected}. The analysis of Eq.~\eqref{exp sig, def} is considerably challenging, for its coefficients become singular as limit $m\to 0^+$. Furthermore, even for $p =0$, our results cannot be recovered simply by extending the degree-2 result in \cite{friz2015physical}.\footnote{The direct generalisation of \cite{friz2015physical} to arbitrary tensor degree suggests that $\lim_{m\to 0^+}\Phi_{n}^{(m)}(0, t)$ is the signature of a 2-rough path. This, however, differs from our result by a factor of $n!$ for each degree $2n$.} Also note that \cite{BPQ_2013} is not applicable to our case, as $\left\{P_t\right\}_{t\geq 0}$ in Eq.~\eqref{exp sig, def} is not centred in general. 
\end{remark}

\subsection{Organisation} In \S\ref{sec: prelim}, we collect some preliminaries on It\^o diffusion processes and the recursive PDE system satisfied by the expected signature. The small mass limit for the 1D case is completely solved via the explicit formulae in \S\ref{sec: expsig of phys BM}. The most important section is \S\ref{sec: arb dim}, in which the small mass limit for the expected signature of the momentum of physical Brownian motion in \emph{arbitrary dimension} is carefully justified, by delicate analyses on the associated recursive PDE system. Finally, when $\scrm$ is diagonalisable, we obtain explicit expressions for the small mass limit in \S\ref{sec: diag}. The proofs of several technical lemmas are given in Appendix~\ref{sec: appendix, lemma}. Detailed computations for the third level expected signature $\Phi_3^{(m)}$ can be found in the \hyperlink{https://github.com/DeepIntoStreams/EsigPhysicalBM/blob/main/Phi_3_computation.pdf}{online appendix}.

\section{Preliminaries}\label{sec: prelim}
\subsection{It\^{o} diffusion process}


Denote by $W:=\left\{W_t\right\}_{t \geq 0}$ the standard Wiener process in $\R^{d_0}$. Consider an It\^o diffusion process $X$ in $\R^d$ satisfying the SDE:
\begin{equation}\label{SDE}
\begin{split}
\dd X_{t}^{(j)} = \mu^{(j)}(t, X_{t})\,\dd t \,+ \,&\sum_{i = 1}^{d_0} V_{i}^{(j)}(t, X_{t}) \cdot \dd W_{t}^{(i)},\quad X_{0}^{(j)} = x^{(j)},\\
& j \in \{1,2,\ldots,d\},\text{ and } t \in [0,T].
\end{split}
\end{equation}
Here $\mu = \left\{\mu^{(j)}\right\}_{j = 1}^{d}$ and $V=\left\{V_{i}^{(j)}\right\}_{1\leq i\leq d_0, 1\leq j \leq d}$ are measurable functions, called the drift vector and the dispersion matrix, respectively. The integral is  in the It\^o sense. Equivalently,
\begin{eqnarray*}
\dd X_{t} = \mu(t, X_{t})\,\dd t + V(t, X_{t}) \cdot \,\dd W_{t}.
\end{eqnarray*}
In addition, $B:=VV^\top$ is known as the diffusion matrix, whose entries are
\begin{eqnarray*}
b_{j_{1}, j_{2}}(x) = \sum_{i = 1}^{d_0}V^{(j_{1})}_{i}(x)V^{(j_{2})}_{i}(x),\qquad  j_{1}, j_{2} \in \{1, 2, \dots, d\}.
\end{eqnarray*}
If $\mu$ and $V$ are functions of $x$ only, then $X$ is said to be a \emph{time-homogeneous} process.

A continuous stochastic process $X:=(X_t)_{t \in [0, T]}$ defined on the probability space $(\Omega,\mathcal{F},\mathbb{P})$ is said to be a strong solution for Eq.~\eqref{SDE} if the following holds (Pavliotis \cite[Definition 3.2]{Pavliotis}):
\begin{enumerate}
    \item $X$ is \emph{a.s.} continuous and adapted to the filtration generated by $W$; 
    \item $\mu(\bullet, X_\bullet)\in L^1\left([0,T], \R^d\right)$ and $V\left(\bullet,X_\bullet\right)\in L^2\left([0,T],  L^2\left(\mathbb{R}^{d}, \mathbb{R}^{d_0}\right)\right)$;
    \item for every $t \geq 0$, Eq.~\eqref{SDE} holds \emph{a.s.}. 
\end{enumerate}

The following criterion for the existence and uniqueness of the strong solution to Eq.~\eqref{SDE} is classical (\cite[p.64]{Pavliotis}). For \emph{time-homogeneous} processes, the second condition  in Proposition~\ref{propn: SDEUniqueSolution} follows from the first one, which may not be the case for general stochastic processes.
\begin{proposition}\label{propn: SDEUniqueSolution}
The SDE \eqref{SDE} has a unique strong solution $\left\{X_t\right\}_{t\in [0,T]}$ whenever there exists $C>0$  such that for every $x, y \in \R^d$ and $t\in[0,T]$, one has $       |\mu(x)-\mu(y)|+ \|V(x)-V(y)\|_F\leq C|x-y|$, and for all $x \in \R^d$ and $t\in[0,T]$  one has $
|\mu(t,x)|+\|V(t,x)\|_F\leq C(1+|x|)$. \footnote{
Here $\|A\|_F:=\sqrt{{\rm tr}\left(A^\top A\right)}$ is the Hilbert--Schmidt (Frobenius) norm of matrices. It is different from the $\ell^\infty$-norm of tensors $\|\bullet\|$, but is equivalent to $\|\bullet\|$ modulo a dimensional constant. See Convention~\ref{convention} below.}
\end{proposition}

An important special case of It\^{o} diffusion process is the class of OU (Ornstein--Uhlenbeck) processes, whose SDE is as follows:
\begin{equation}\label{OU process, new}
    \dd X_t = -\theta X_t \,\dd t + \sigma \,\dd W_t,
\end{equation}
where $\theta$ and $\sigma$ are real parameters. The physical Brownian motion falls into this category.


\subsection{Expected signature}
We collect some basic notations and results in rough path theory.

\begin{definition}
A formal $\R^d$-tensor series is a sequence $\left\{a_n\in \left(\R^d\right)^{\otimes n}\right\}_{n\in\mathbb{N}}$, written in the sequel as ${\bf a} = (a_0, a_1, \dots)$. Set ${\bf 1} := (1,0,\dots)$ and ${\bf 0} := (0,0,\dots)$.  The space of  formal $\R^d$-tensor series is denoted as $\tens$. The addition and tensor product of $\R^d$-tensor series are as follows: for ${\bf a} = (a_0, a_1, \dots)$ and ${\bf b} = (b_0, b_1, \dots)$, set ${\bf a}+{\bf b} := (a_0+b_0, a_1+b_1, \dots)$ and ${\bf a}\otimes{\bf b} = {\bf c}$, where $c_n = \sum_{i=0}^n a_i b_{n-i}$ for each $n \in \mathbb{N}$. The scalar multiplication is $\lambda {\bf a} := (\lambda a_0, \lambda a_1,...)$ for $\lambda \in \mathbb{R}$.
\end{definition}

The space $\tens$ becomes an associative unital algebra over $\mathbb{R}$ when  equipped with $+$, $\otimes$. An element ${\bf a}\in E$ is invertible iff $a_0\neq 0$. Its inverse is given by ${\bf a}^{-1} = \frac{1}{a_0}\sum_{n\geq 0}\left({\bf 1}-\frac{{\bf a}}{a_0}\right)^n$.

\begin{definition} \label{def, sig}
Let $\left\{X_t\right\}_{t \geq 0}$ be an It\^o diffusion process governed by SDE~\eqref{SDE}. 
The signature \\
$S(X_{[0,t]}):= \left(1, S^1(P_{[0,t]}), S^2(P_{[0,t]}), \cdots\right)\in \tens$ is defined by the Stratonovich iterated integrals for almost every path $X$ and for every $n \geq 1$:
\begin{equation*}
S^n(X_{[0,t]}):=\idotsint\limits_{0\leq t_1<\dots<t_n\leq t}\dd X_{t_1}\otimes\cdots\otimes\dd X_{t_n}.
\end{equation*} 
\end{definition}

\begin{definition}\label{definition:ExpSig}
Let $\left\{X_t\right\}_{t \geq 0}$ be an It\^o diffusion process governed by SDE~\eqref{SDE} with the initial data $X_0 = x$. Suppose that the signature $S(X)$ is well-defined. The expected signature of $X_{[0,t]}$ is
\begin{equation*}
    \Phi(t,x) := \E^x\left[S\left(X_{[0,t]}\right)\right].
\end{equation*}
The level-$n$ truncation of the expected signature for $n\in\mathbb{N}$ is $\Phi_n(t,x)=\pi_n \left[\Phi(t,x)\right]$, where $\pi_n: \tens \to \bigoplus_{i=0}^n \left(\R^d\right)^{\otimes i}$ is the natural projection. (Note that $\Phi_0(t,x) \equiv 1$.)
\end{definition}

As our primary example, the expected signature of the
Brownian motion in arbitrary dimension $d$ up to a fixed time $T$ is known (Fawcett \cite{Fawcett_2003}): Let $\{W_t\}_{t \geq 0}$ be a standard $d$-dimensional Brownian motion starting from 0. Then it holds that
\begin{align*}
    \E\left[S\left(W_{[0,T]}\right)\right] = \exp\left(\frac{T}{2}\sum_{i=1}^d e_i\otimes e_i\right)=\left(1, \; 0, \; \left( \sum_{i = 1}^{d}\frac{T}{2} e_i \otimes e_i\right), \; 0, \; \frac{1}{2!}\left(\sum_{i = 1}^{d}\frac{T}{2} e_i \otimes e_i\right)^{\otimes2},\, \ldots\right).
\end{align*} 

\subsection{PDE  for expected signature and a Feynman--Kac type result}

The expected signature $\Phi$ of a path $X_t$ can be identified with the strong solution for a graded system of parabolic PDE that admits a representation formula~\eqref{Feynman-Kac, rep formula}.

\begin{theorem}\label{thm: PDE}
Let $\left\{X_t\right\}_{t \geq 0}$ be an It\^{o} diffusion process governed by SDE~\eqref{SDE} with expected signature $\Phi: [0,T]\times\R^d\to\tens$ (Definition~\ref{definition:ExpSig}). Assume for each $n \in \mathbb{N}$ that $$\pi_n\Phi\in C^1\left([0,T];C^2 \cap L^\infty \left(\R^d; \bigoplus_{i=0}^n \left(\R^d\right)^{\otimes i}\right)\right).$$  Then $\Phi$ verifies the system of PDE:
\begin{footnotesize}
\begin{eqnarray*}
\left(-\frac{\partial }{\partial t}+ \mathcal{A}\right)\Phi(t,x) +\sum_{j=1}^{d}\left(\sum_{j_{1}=1}^{d}b_{j_{1},j}(x)e_{j_{1}}\right) \otimes \frac{\partial \Phi(t,x)}{\partial x_{j}}+\left(\sum_{j = 1}^{d} \mu^{(j)}(x)e_{j}+\frac{1}{2}\sum_{j_{1}=1}^{d}\sum_{j_{2}=1}^{d}b_{j_{1},j_{2}}(x)e_{j_{1}}\otimes e_{j_{2}}\right) \otimes \Phi(t,x)=0,
\end{eqnarray*}
\end{footnotesize}
subject to the conditions $    \Phi(0, x) = \mathbf{1}$ and $\Phi_{0}(t, x) = 1$ for all $(t, x) \in [0,T] \times \R^d$. Here 
$\mathcal{A}= \sum_{j=1}^{d}\mu^{(j)}(x)\frac{\partial}{\partial x_{j}}+\frac{1}{2}\sum_{j_{1}=1}^{d}\sum_{j_{2}=1}^{d}b_{j_{1},j_{2}}(x)\frac{\partial^{2}}{\partial x_{j_{1}}\p x_{j_{2}}}$ is the infinitesimal generator of $\left\{X_t\right\}_{t \geq 0}$.


Conversely, assume that the above PDE system has a solution $\Phi$ such that $$\pi_n\Phi\in C^1\left([0,T];C^2 \cap L^\infty \left(\R^d; \bigoplus_{i=0}^n \left(\R^d\right)^{\otimes i}\right)\right)$$ for each $n \in \mathbb{N}$. Moreover, $\Phi(t,p)$ and its first derivative satisfy the polynomial growth condition. Then $\Phi$ agrees with the expected signature of the It\^{o} diffusion process Eq.~\eqref{SDE}.

\end{theorem}
\begin{proof}
See \cite[p.49]{ni2012expected} for a  proof  based on martingale properties. \end{proof}

Theorem~\ref{thm: PDE} states that the expected signature of an It\^{o} diffusion process, \emph{viewed as a function of the initial datum and the terminal time $t$}, is equivalent to the strong solution --- \emph{i.e.}, $C^1$ in time and $C^2 \cap L^\infty$ in space at each level of truncation --- to a graded (nested or recursive) PDE system; namely, a $\tens$-valued PDE. Notice that each $\Phi_n$ depends on $\Phi_{n-1}$ and $\Phi_{n-2}$, so the PDE system should be solved iteratively. The strong solution is clearly unique if it exists.

We now introduce our central tool, a Feynman--Kac theorem \cite[p.147]{Friedman}, which connects stochastic processes to parabolic PDEs.
\begin{lemma}\label{lem:FeynmanKac}
Let $\left\{X_t\right\}_{t \geq 0}$ be an It\^{o} diffusion process on $\R^d$ described by Eq.~\eqref{SDE}:
\begin{align*}
    \dd X_{t}^{(j)} = \mu^{(j)}(X_{t})\,\dd t + \sum_{i = 1}^{d_0} V_{i}^{(j)}(X_{t}) \cdot \dd W_{t}^{(i)} \qquad \text{for each $j \in \{1,\ldots,d\}$.}
\end{align*}
Denote by $\mathcal{A}$ the infinitesimal generator with initial datum $X_0=x$ and fix $T>0$. (The dependence on $T$ is suppressed in the sequel.) Consider  for $u: [0,T] \times \R^d \to \R$ the Cauchy problem
\begin{equation}\label{cauchy prob}
\begin{cases}
    \partial_tu- \mathcal{A}u =  f \qquad &\text{ in } [0,T] \times \R^d, \\
    u= 0 \qquad &\text{ at } t=0.
    \end{cases}
\end{equation}
Assume the following conditions for $V=\left\{V_i^{(j)}\right\}$, $\mu=\left\{\mu^{(k)}\right\}$, and $f$:
\begin{enumerate} 
\item There exists $c_0=c(V)>0$ such that 
\begin{align*}
\sum_{i,j=1}^d b_{i,j}(x)\xi_i\xi_j \geq c_0|\xi|^2 \quad \text{for each } \xi \in \R^d \text{ and for all } x \in \R^d.
\end{align*}
\item 
$\mu^{(i)}$, $V_{i}^{(j)}$ are bounded and $f$ is continuous on $[0,T] \times \R^d$.
\item
$\mu^{(i)}$ and $V_{i}^{(j)}$ are uniformly Lipschitz in $(t, x)$ in compact subsets of $[0,T] \times \R^d$. That is, for any compact $\mathcal{K} \Subset [0,T] \times \R^d$, there is $C_1 = C(\mathcal{K},d,\mu, V)$ such that for any $i,k\in \{1,\ldots,d\}, \;j\in \{1,\ldots,d_0\}$,
\begin{align*}
\sup_{(t, x),(s, y) \in \mathcal{K};\, (t, x) \neq (s, y) }  {\left|\mu^{(i)}(t, x)-\mu^{(i)}(s, y)\right| + \left|V_{j}^{(k)}(t, x)-V_{j}^{(k)}(s, y)   \right|} \leq  C_1 \left(|x-y| + |t-s|\right).
 \end{align*}    
\item
$V_i^{(j)}$ are H\"{o}lder in $x$ uniformly: there exist $\gamma = \gamma(V) \in \; ]0,1]$, $C_2 = C(d,V,\gamma)$ such that
\begin{align*}
    \sup_{x \neq y \text{ in } \R^d;\, t \in [0,T]} \frac{\left|V_{j}^{(k)}(x)-V_{j}^{(k)}(y)   \right|}{|x-y|^\gamma} \leq  C_2
 \end{align*}
for any $j\in \{1,\ldots,d_0\},\; k\in \{1,\ldots,d\}$. 
\item $f$ is  H\"{o}lder in $x$ uniformly: there exist $\eta =\eta(f) \in ]0,1]$, $C_3=C(d,V,\eta)$ such that
\begin{align*}
\sup_{x \neq y \text{ in } \R^d;\, t \in [0,T]} \frac{\left|f(t,x)-f(t,y)   \right|}{|x-y|^\eta} \leq  C_3.
\end{align*}
\item In addition, $f$ satisfies a power law growth condition: there are positive constants $C_4=C(f,d)$ and $c_5=c(f,d)$ such that 
$|f(x,t)|\leq C_4(1+|x|^{c_5}) \qquad \text{for all } (x,t) \in [0,T] \times \R^d.$
\end{enumerate} 
Under these assumptions, there exists a unique solution $u \in C^1\left([0,T];C^2\left(\R^d\right)\right)$ to the Cauchy problem~\eqref{cauchy prob},  given by the representation formula
\begin{equation}\label{Feynman-Kac, rep formula}
    u(t, x) = \E^x\left[\int_0^tf(t-s, X_s) \,\dd s\right].
\end{equation}
\end{lemma}


\section{Expected signature of physical Brownian motion via the graded PDE system}\label{sec: expsig of phys BM}


\subsection{PDE for physical Brownian motion}

Let $\Phi = \{\Phi_n\}_{n \in \mathbb{N}}: [0,T] \times \R^d \to \tens$ be the expected signature of $\{P_t\}$, the momentum of the Brownian particle in magnetic field governed by the SPDE (recall Eq.~\eqref{p eqn}): $
\dd P_t = -\mom P_t \,\dd t +\dd W_t$ with $P\big|_{t=0} = p$. \footnote{From now on, write $(t, P)$ for a generic point in the domain $[0,T]\times\R^d$ of $\Phi$. The symbol $p \in \R^d$ is reserved for the initial momentum.}

In view of Theorem~\ref{thm: PDE}, $\Phi$ satisfies the graded PDE system
\begin{align}\label{PDE for Phi, later}
    &\left(-\p_t + \mathcal{A}\right)\Phi_n(p,t)   = \mmp \otimes \Phi_{n-1}(p,t) -\sum_{j=1}^d e_j \otimes \p_{p_j}\Phi_{n-1}(p,t) \nonumber\\
    &\qquad\qquad - \frac{1}{2}\sum_{j=1}^d e_j \otimes e_j \otimes \Phi_{n-2}(p,t) =: - \fs_n(p,t)
\end{align}
for each $n \in \mathbb{N}$ with $\Phi_0 = {\bf 1} = (1,0,0,\ldots)$. The infinitesimal generator on the left-hand side is $$
    \mathcal{A} = \frac{1}{2}\Delta - \mmp\cdot\na,$$ 
where  $\na$ and $\Delta$  are  taken with respect to $p$. In addition, the diffusion matrix $B$ is the identity matrix, and the drift vector $\mu=-\mmp$.

\begin{remark}
One crucial observation is that, as we are working with the PDE for the expected signature (rather than the signature), all the parameters $\mu, \mathcal{A}, B$ of the It\^o diffusion process depend only on the initial datum $p$. Thus,  here and hereafter, we always view $\Phi$ as a function of $(t,p)$.
\end{remark}

The Feynman--Kac-type formula in Lemma~\ref{lem:FeynmanKac} allows us to express 
\begin{align}\label{feynmann-kac, later}
\Phi_n(t,p) \equiv \Phi_n^\mm(t,p) = \E^p\left[ \int_0^t \mathfrak{S}_n\left(t-s, P_s \right)\,\dd s \right].
\end{align}
This serves as the beginning point of our analysis on the small mass limit of expected signature of the momentum in the subsequent parts of the paper.

\subsection{The case of one space dimension}

We embark on the analysis on the small mass limit of $\Phi_n^\mm(t,p)$ from the simplest case $d=1$. In this case, the expected signature $\Phi_n$ is a real number for each $n=0,1,2,\ldots$, which significantly simplifies the analysis. 

The expected signature in 1D for any OU process can be explicitly computed:

\begin{lemma}\label{lemma:SolutionOU_1D}

Let $\{X_s\}_{s\in [0,t]}$ be a 1-dimensional OU process:
\begin{equation*}
\dd X_t = -\theta X_t \,\dd t + \sigma \,\dd W_t
\end{equation*}
for real constants $\sigma$ and $\theta$. Its expected signature $\Phi(t,x) := \E\left[S(X_t)\big|X_0=x\right]$ is given by\footnote{For $n \in \mathbb{Z}_{\geq -1}$, $n!! =
    \prod_{k=0}^{\frac{n+1}{2}}(2k-1)$ if  $n$  is odd, $\prod_{k=1}^{\frac{n}{2}}(2k)$ if $n$ is even, and $1$ if $n=0$ or $-1$.}
\begin{eqnarray*}
&&  \Phi_{2k+1}(t,x) = \frac{1}{(2k+1)!}\Bigg\{\sum_{i=0}^k {2k+1 \choose 2i}(-x)^{2k+1-2i}\left(1-e^{-\theta t}\right)^{2k+1-2i}\left(\frac{\sigma^2}{2\theta}\right)^i\left(1-e^{-2\theta t}\right)^i(2i-1)!!\Bigg\},\\
&&  \Phi_{2k}(t,x) = \frac{1}{(2k)!}\Bigg\{\sum_{i=0}^{k}{2k \choose 2i}(-x)^{2k-2i}\left(1-e^{-\theta t}\right)^{2k-2i}\left(\frac{\sigma^2}{2\theta}\right)^i\left(1-e^{-2\theta t}\right)^i(2i-1)!!\Bigg\}.
\end{eqnarray*}
\end{lemma}

\begin{proof}[Proof of Lemma~\ref{lemma:SolutionOU_1D}]
   See \cite[Lemma 4.6.1]{ni2012expected}. 
\end{proof}

Specialised to the problem of physical Brownian motion in 1-D, the SDE is $\dd P_t = -\frac{\scrm}{m}P_t \,\dd t + \dd W_t$ with initial datum $P\big|_{t=0}=p$. Lemma~\ref{lemma:SolutionOU_1D} with $\theta = \frac{\scrm}{m}$ and $\sigma = 1$ gives us the explicit solution to the expected signature $\Phi=\left\{\Phi_n\right\}$ of the momentum:
\begin{footnotesize}
\begin{eqnarray*}
&&  \Phi_{2k+1}(t,p) = \frac{1}{(2k+1)!}\Bigg\{\sum_{i=0}^k {2k+1 \choose 2i}(-p)^{2k+1-2i} \left(1-e^{-\frac{\scrm}{m} t}\right)^{2k+1-2i}\left(\frac{m}{2\scrm}\right)^i\left(1-e^{-2\frac{\scrm}{m} t}\right)^i(2i-1)!!\Bigg\},\\
&&  \Phi_{2k}(t,p) = \frac{1}{(2k)!}\Bigg\{\sum_{i=0}^{k}{2k \choose 2i}(-p)^{2k-2i}\left(1-e^{-\frac{\scrm}{m} t}\right)^{2k-2i}\left(\frac{m}{2\scrm}\right)^i\left(1-e^{-2\frac{\scrm}{m} t}\right)^i(2i-1)!!\Bigg\}.
\end{eqnarray*}

\end{footnotesize}

 
Observe that $\Phi_n$ is a degree-$n$ polynomial in $p$ consisting of only odd or even terms in accordance with the parity of $n$. In the limit $m \to 0^+$, except for the leading term $\frac{(-p)^n}{n!}$  in $\Phi_n (t,p)$ (with the top degree in $p$), all the other terms tend to zero at the rate $\lesssim \mathcal{O}(m)$.

\begin{proposition}\label{propn: 1D}
The expected signature $\Phi$ of a 1D Brownian particle  subject to the SDE
\begin{equation*}
\dd P_t = -\frac{\scrm}{m}P_t \,\dd t + \dd W_t,\qquad P\big|_{t=0}=p
\end{equation*}
has the following zero-mass limit: for any $p \in \R$, $t \in \R_+$, and $\scrm \in \R$, 
\begin{equation*}
\lim_{m \to 0^+}  \Phi(t,p) = \left(1, -p, \frac{(-p)^2}{2!}, \cdots, \frac{(-p)^n}{n!}, \cdots\right).
\end{equation*}
\end{proposition}

\subsection{Higher dimensional case foreshadowed}

The above result for 1-D generalises to multiple dimensions in a highly non-trivial way: non-commutative effects become transparent for any $d \geq 2$. As a primer of the developments in the next section, let us point out that the first two levels of the expected signature in higher dimensions take the following forms: 
\begin{align*}
    \Phi_1^{d\geq 2}(t,p) &= \underbrace{\int_0^{\frac{t}{m}}\left(-\scrm e^{-\scrm t_1}\right)p\; \dd t_1}_{:= \fraka_1^\mm(t,p)} ;\\
    \Phi_2^{d\geq 2}(t,p) &=\underbrace{\idotsint\limits_{0\leq t_1<t_2\leq \frac{t}{m}}\left(-\scrm e^{-\scrm t_1}p\right)\otimes\left(-\scrm e^{-\scrm t_2}p\right)\dd t_1 \dd t_2}_{:=\fraka_2^\mm(t,p)}+\underbrace{\left(\scrm \cdot \dashint_0^{\frac{t}{m}} \ssig_\sigma\,\dd \sigma - \frac{\id}{2}\right) t}_{:= \frakb_2^\mm(t)} + {\rm Error}_2^\mm(t),
\end{align*}
where $\ssig_\sigma = \int_0^\sigma e^{-\scrm \varsigma} e^{-\scrm^*\varsigma}\,\dd\varsigma$ and $\left\|{\rm Error}_2^\mm\right\| \lesssim \mathcal{O}(m)$ as $m \to 0^+$.

When $d=1$, the term $\fraka_1^\mm(t,p)$ reduces to its 1D counterpart $(-p)\left(1-e^{-\frac{\scrm}{m} t}\right)$ derived in Proposition~\ref{propn: 1D}. Things are similar for the second level of expected signature: $\fraka_2^\mm(t,p)$ reduces to $\frac{(-p)^2}{2}\left(1-e^{-\frac{\scrm}{m} t}\right)^2$ when $d=1$, but the following direct extension is, in general, \emph{false}:
\begin{align*}
\fraka_2^\mm(t,p) = \frac{1}{2}\left(-p+e^{-\frac{\scrm}{m} t}p\right)^{\otimes 2} \qquad \text{ for } d \geq 2.
\end{align*} 
The failure is due to the asymmetry of $e^{-\scrm t_1}$ and $e^{-\scrm t_2}$.

When $d\geq 2$, the additional term  $\frakb_2^\mm(t)$ becomes indispensable in the small mass limit:
\begin{align*}
\frakb_2^\mm(t,p) = \left(\scrm \cdot \dashint_0^{\frac{t}{m}} \ssig_\sigma\,\dd \sigma - \frac{\id}{2}\right) t =\frac{1}{2} \left(\scrm \cdot \dashint_0^{\frac{t}{m}} \ssig_\sigma\,\dd \sigma - \dashint_0^{\frac{t}{m}} \ssig_\sigma\,\dd \sigma\cdot \scrm^*\right) t.
\end{align*}
Subsequent analyses reveal that this term is skew-Hermitian/anti-symmetric; hence, in the 1D case (where $\fraka_2^\mm$ is a real number) it must be zero.

The identification and proof for the small mass limit of the expected signature $\Phi_n$ in arbitrary dimensions is the main objective of the next section. 

\section{Small mass limit in arbitrary dimensions}\label{sec: arb dim}

Consider the physical Brownian motion in dimensions $d = 1,2,3,\ldots$. Recall the PDE for $\Phi_n^\mm(t,p)$ and the representation formula (Eqs.~\eqref{PDE for Phi, later}, \eqref{feynmann-kac, later}). We shall write $\Phi_n \equiv \Phi_n^\mm$ to emphasise the dependence on $m$, and similarly for any other variables. Our goal is to identify
\begin{align*}
    \lim_{m \to 0^+} \Phi^\mm_n(t,p) &\quad \text{for each $n = 1,2,3,\ldots$ in the $\ell^\infty$-tensor norm on $\left(\R^d\right)^{\otimes n}$}.
\end{align*}

\begin{convention}\label{convention}
Throughout this work, we take the $\ell^\infty$-norm on the tensor algebra $\tens$. That is, for each $n \in \mathbb{N}$ and any $\mathcal{T} \in \left(\R^d\right)^{\otimes n}$, express $\mathcal{T}$ uniquely as a linear combination of the canonical basis and take the supremum of absolute value of all the coefficients. The symbol $\|\mathcal{T}\|$ is reserved for the $\ell^\infty$-norm of tensors, unless otherwise indicated. It makes $\tens$ a Banach algebra: $\|\mathcal{T}\otimes\mathcal{S}\| \leq \|\mathcal{T}\| \|\mathcal{S}\|$ for any finite-rank tensors $\mathcal{T}$ and $\mathcal{S}$.

The $\ell^\infty$-norm is only taken for simplicity. Indeed, one may take  $\ell^r$-norm for any $r \in [1,\infty]$, as the $\ell^r$- and $\ell^\infty$-norms of $\mathcal{T} \in \left(\R^d\right)^{\otimes n}$ are equivalent modulo a constant $C$ depending only on $n$, $d$, and $r$. In particular, $C$ is independent of $p$ and $m$, hence is harmless for  the small mass limit. We refer to Ni--Xu \cite[pp.3--4]{nx} for the choice of tensor norms. \footnote{Here we inevitably encounter a conundrum: for $d\times d$ matrices $A$, $B$ viewed as rank-2 tensors, $\|AB\| \leq d\|A\|\|B\|$; the factor $d$ is in general indispensable. Thus, $\mathfrak{gl}(d;\C)$ under the $\ell^\infty$-norm  of tensors does \emph{not} form a Banach algebra under the matrix product. The reason lies in the discrepancy between tensor product and matrix product: the latter is a contraction of the former via taking trace, hence is not a natural product on the tensor algebra.}

\end{convention}

\subsection{Main theorem}\label{subsec: main thm}
Define for $n \in \{1, 2, 3, \cdots\}$:
\begin{align}\label{leading term fraka def}
\fraka_n^\mm(t,p) &:=\idotsint\limits_{0\leq t_1<\dots<t_n\leq t}\left(-\frac{\scrm}{m}e^{-\frac{\scrm}{m}t_1}p\right)\otimes\cdots\otimes\left(-\frac{\scrm}{m}e^{-\frac{\scrm}{m}t_n}p\right)\,\dd t_1\ldots\,\dd t_n\nonumber\\
&= \idotsint\limits_{0\leq t_1<\dots<t_n\leq t} \,\,\bigotimes_{j=1}^n  \left( -\mom e^{-\mom t_j} p\right)\, \dd t_1\ldots\,\dd t_n.
  \end{align}
Note here the non-commutative nature of the integrand. By convention we set $\Phi_0^\mm \equiv 1$ and $\fraka_0^\mm \equiv {\bf 1} :=(1,0,0,\ldots). $ Also define
\begin{align*}
\ssig_\sigma := \int_0^\sigma e^{-\scrm \varsigma} e^{-\scrm^*\varsigma}\,\dd\varsigma.
\end{align*}

Our main result of this paper, Theorem~\ref{thm: main, Dec24}, is reproduced as Theorem~\ref{main theorem for singular limit} and Corollary~\ref{cor, explicit limit} below. The latter presents the explicit formula for $\lim_{m \to 0^+} \Phi_n^{(m)}$, and the former gives the asymptotics for $\Phi_n^{(m)}$ with small $m>0$. In this subsection we shall deduce Corollary~\ref{cor, explicit limit} from Theorem~\ref{thm: main, Dec24}, while the proof of Theorem~\ref{thm: main, Dec24} occupies the remaining parts of the paper.

\begin{theorem}\label{main theorem for singular limit} Let $\Phi_{n}^\mm: \R^d\times\R_+ \to \left(\R^d\right)^{\otimes n}$ be the degree-$n$ expected signature of the momentum path $\left\{P_t\right\}$. Suppose  $\|\scrm\|\leq \Lambda$ and  there is some $K>0$ such that $\max\left\{\left\|e^{-\scrm^*\delta}\right\|,\,\left\|e^{-\scrm\delta}\right\|\right\} \leq K e^{- \lambda \delta}$ for each $\delta>0$,\footnote{See Lemma~\ref{lem: matrix expoential limit} for a justificiation of this assumption.} where $\lambda \leq\Lambda$ and $K$ are finite positive constants depending only on $\scrm$. Then for each $n = 1,2,3,\ldots$ we have
\begin{align*}
\Phi_n^\mm(t,p) &= \sum_{k=0}^{\lfloor {n}\slash{2}\rfloor} \left\{\fraka_{n-2k}^\mm(t,p) \otimes \left[ \left(\scrm \cdot\dashint_0^{{t}\slash{m}}\ssig_\sigma \,\dd \sigma -\frac{\id}{2}\right)t\right]^{\otimes k}  \right\} +  {\rm Error}_n^\mm(t,p),
\end{align*}
where ${\rm Error}_n^\mm(t,p)$ is a polynomial of degree $\leq \max\{n-2,0\}$ in $p$. Also,
\begin{align*}
\left\|{\rm Error}_n^\mm(t,p)\right\| \leq C\left(K, \Lambda, \lambda^{-1}, t, d, n\right)\, \left(1+|p|^{\max\{n-2,0\}}\right)m.
\end{align*}
In particular, this error term vanishes in the massless limit $m\to 0^+$ for each fixed $n,d \in \mathbb{N}$, $t \in ]0,\infty[$, $p \in \R^d$, and matrix $\scrm$ as above.  
\end{theorem}

\begin{remark}
We adopt the convention $\fraka_0^\mm(t,p) \equiv \fraka_0 = 
\frac{1}{n!} \1 = \left(\frac{1}{n!},0,0,\ldots\right).$  Theorem~\ref{main theorem for singular limit} shows that when $n$ is even, at $p=0$ we have
\begin{align*}
\Phi_n^\mm(t,0) = \left[ \scrm \cdot \dashint_0^{t/m} \ssig_\sigma\,\dd\sigma - \frac{\id}{2}\right]^{\otimes n/2} t^{n/2} + {\rm Error}_n^\mm(t,0).
\end{align*}
\end{remark}

As an immediate corollary of Theorem~\ref{main theorem for singular limit}, for each $q \in \mathbb{N}$ denote
\begin{align*}
    \overline{\fraka}_{q}(p) := \idotsint\limits_{0\leq t_1<\dots<t_q < \infty}\left(-\scrm e^{-\scrm t_1}p\right)\otimes\cdots\otimes\left(-\scrm e^{-\scrm t_q}p\right)\,\dd t_1\ldots\,\dd t_q,\qquad \scrc := \ssig_\infty = \int_0^\infty e^{-\scrm t}e^{-\scrm^* t}\,\dd t.
\end{align*}

\begin{corollary}\label{cor, explicit limit}
In the setting of Theorem~\ref{main theorem for singular limit}, it holds that
\begin{align*}
\lim_{m \to 0^+} \Phi^\mm_n(t,p) &= \sum_{k=0}^{\lfloor {n}\slash{2}\rfloor} \left\{\overline{\fraka}_{n-2k}(p) \otimes  \left(\scrm \scrc -\frac{\id}{2}\right)^{\otimes k}  \right\}t^k = \sum_{k=0}^{\lfloor {n}\slash{2}\rfloor} \left\{\overline{\fraka}_{n-2k}(p) \otimes  \left(\frac{\scrm \scrc - \scrc\scrm^*}{2}\right)^{\otimes k}  \right\}t^k
\end{align*}
in the $\ell^\infty$-norm of tensors. In particular, the massless limit at level $n$ is a polynomial of degree $n$ in $p$ and of degree $\lfloor {n}\slash{2}\rfloor$ in $t$. 
\end{corollary}

\begin{proof}[Proof of Corollary~\ref{cor, explicit limit}] 
In light of Theorem~\ref{main theorem for singular limit}, it suffices to show that
\begin{align*}
\fraka_q^\mm(t,p) \longrightarrow \overline{\fraka}_q(p)\qquad \text{and}\qquad\scrm\cdot\dashint_0^{{t}\slash{m}}\ssig_\sigma \,\dd \sigma -\frac{\id}{2} \longrightarrow \frac{\scrm\scrc - \scrc\scrm^*}{2}
\end{align*}
as $m\to 0^+$ for each $q \in \mathbb{N}$. In fact, these limits holds uniformly in $t$. 

The first limit is an immediate consequence of change of variables $\tau_j := t_j \slash m$ in the definition of $\fraka^\mm_q$. For the second one, it is enough to prove that 
\begin{align}\label{C matrix}
\lim_{T \to \infty} \dashint_0^{T} \ssig_\sigma\,\dd\sigma = \scrc.
\end{align}
This is because $A:=\scrm\scrc - \frac{\id}{2}$ is a skew-Hermitian matrix (\emph{cf.} Friz--Lyons--Gassiat \cite[p.7945]{friz2015physical}), so $A = \frac{A-A^*}{2} = \frac{\scrm\scrc - \scrc\scrm^*}{2}$ where $\scrc$ is  Hermitian. In Eq.~\eqref{C matrix} we take $T:=t\slash m \to \infty$ as $m \to 0^+$.

To see Eq.~\eqref{C matrix}, recall $\ssig_\sigma := \int_0^\sigma \scrg(s)\,\dd s$ with Hermitian matrix-valued function $\scrg(s):=e^{-\scrm s}e^{-\scrm^*s}$. By Fubini's theorem, $\dashint_0^{T} \ssig_\sigma\,\dd\sigma  
= \frac{1}{T} \int_0^T \int_s^T \scrg(s)\,\dd\sigma\,\dd s = \int_0^T \scrg(s)\,\dd s - \frac{1}{T} \int_0^T s\scrg(s)\,\dd s$. Integration by parts and the assumption $\max\left\{\left\|e^{-\scrm^*\delta}\right\|,\,\left\|e^{-\scrm\delta}\right\|\right\} \leq K e^{-\lambda\delta}$ for each $\delta>0$ then yield that $\left\|\dashint_0^{T} \ssig_\sigma\,\dd\sigma - \int_0^T \scrg(s)\,\dd s\right\| \leq \frac{K^2d}{4\lambda^2T} \to 0$ as $T \to \infty$. Thus Eq.~\eqref{C matrix} follows.   \end{proof}

For a (tensor-valued) function $f$, we set 
 $  \bra f \ket :=  \lim_{T\to\infty}\frac{1}{T} \int_0^T f(t)\,\dd t. $
It is shown along the proof of Corollary~\ref{cor, explicit limit} that 
  $  \bra \ssig_\bullet\ket  = \ssig_\infty = \int_0^\infty e^{-\scrm t}e^{-\scrm^*t}\,\dd t =: \scrc.$
That is, the large-time average of $\ssig$ coincides with its limit, which is known to exist. This is a recurring theme in ergodic theory, and our computation agrees with Friz--Lyons--Gassiat \cite{friz2015physical}.  

There is the scope to prove the strong convergence in a general case where $p\neq 0$ by extending  \cite{friz2015physical}. It, however, is out of the scope of this paper and is left for future work.
\subsection{Preparatory results}
Several simple lemmas are in order. See Appendix~\ref{sec: appendix, lemma} for the proofs.

\subsubsection{A linear algebra lemma}

The lemma below justifies our assumption on matrix norms in Theorem~\ref{main theorem for singular limit}. 
Recall from Convention~\ref{convention} that all matrix norms are the $\ell^\infty$-norm.

\begin{lemma}\label{lem: matrix expoential limit}
Assume that all the eigenvalues of $\scrm \in \mathfrak{gl}(d;\C)$ have real parts no less than $\lambda>0$. Then, for some constant $K>0$ independent of $\lambda$, one has that $\left\|e^{-\scrm\delta}\right\| \leq K e^{-\lambda \delta}$ for any $\delta>0$.
\end{lemma}

\subsubsection{Averaged integrals}

An elementary lemma will be frequently used later:
\begin{lemma}\label{lem: average difference}
Let $f: [0,\infty[ \to \R$ be a continuous function with $\|f\|_{L^\infty([0,\infty[)} \leq C_0$. Then, for any positive numbers $t, \delta$ such that $t-\delta > 0$, it holds that $\left|\dashint_0^t f - \dashint_0^{t-\delta} f\right| \leq \frac{2 C_0 \delta}{t}.$
\end{lemma}

\subsubsection{Expected signature is a polynomial in the initial momentum} \label{subsec: Phi n polyl in p}
Recall $\Phi_0 = (1,0,0,\ldots)$. Direct calculation (see \S\ref{subsec: Phi 1--4}) gives us $\Phi_1^\mm(t,p) = -\mom \int_0^t e^{-\mom s}p\,\dd s = \left(e^{-\mom t} - \id\right)p.$ Inducting on $n$, we find from Eq.~\eqref{PDE for Phi, later} that the source terms $\fs_n = \fs_n^\mm$ are degree-$n$ polynomials in $p$.  Then by Feynman--Kac formula~\eqref{feynmann-kac, later} we have the following (see \cite[Theorem~4.12]{ni2012expected} for a proof):
\begin{lemma}\label{lem: degree in p}
The $n^{\text{th}}$-level expected signature $\Phi_n^\mm(t,p)$ is a polynomial in $p$ of degree $\leq n$. 
\end{lemma}

In addition,  consider the decomposition $    \mathfrak{S}_n =  \mathfrak{S}^{\rm high}_n  + \mathfrak{S}^{\rm low}_n$, where
\begin{eqnarray*}
&& \mathfrak{S}^{\rm high}_n (t,p):= -\mmp \otimes \Phi_{n-1}(t,p),\\
&& \mathfrak{S}^{\rm low}_n (t,p):=\sum_{j=1}^d e_j \otimes \p_{p_j}\Phi_{n-1}(t,p) + \frac{1}{2}\sum_{j=1}^d e_j \otimes e_j \otimes \Phi_{n-2}(t,p). 
\end{eqnarray*}
Lemma~\ref{lem: degree in p}  implies that $\deg_p \mathfrak{S}^{\rm high}_n -\deg_p  \mathfrak{S}^{\rm low}_n \geq 2$ whenever $n \geq 2$.

\subsubsection{Tensor product}\label{Tensor product of linear mappings}
For vector spaces $V_1, V_2, W_1, W_2$ over the same field and linear mappings $T \in \mathcal{L}(V_1, W_1)$, $S \in \mathcal{L}(V_2, W_2)$, we set $T \otimes S \in \mathcal{L}(V_1 \otimes V_2, W_1 \otimes W_2)$ by $$(T \otimes S)(v\otimes w) = (Tv) \otimes (Sw)$$ for all $v \in V_1$ and $w \in W_1$. This is also referred to as the \emph{mixed product property} of tensor product. If everything is finite-dimensional, it agrees with the \emph{Kronecker product} of matrices.

For $A,B \in \mathfrak{gl}(d;\C)$; $v,w \in \C^d$,  the matrix $(A \otimes B)(v \otimes w) \in \mathfrak{gl}(d;\C)$ has  components 
\begin{align*}
    \big[(A \otimes B)(v \otimes w)\big]^{ij} = \big[(Av) \otimes (Bw)\big]^{ij} = (Av)^i(Bw)^j = \sum_{\alpha,\beta=1}^{d} A^i_\alpha v^\alpha B^j_\beta w^\beta.
\end{align*}
It is more convenient to view $A,B$ as in $\C^d \otimes \left(\C^d\right)^*$, \emph{i.e.}, as type-$(1,1)$-tensors, and  $(A \otimes B)(v \otimes w) \in \C^d \otimes \C^d$, \emph{i.e.}, as a rank-2 contravariant tensor. More generally, given $A,B,Z \in \mathfrak{gl}(d;\C)$ and viewing  $A,B \in \C^d \otimes \left(\C^d\right)^*$; $Z \in \C^d \otimes \C^d$, we define $(A \otimes B) \cdot Z \in \C^d \otimes \C^d \cong \mathfrak{gl}(d,\C)$ by
\begin{align*}
\big[(A \otimes B) \cdot Z\big]^{ij} =\sum_{\alpha,\beta=1}^{d} A^i_\alpha B^j_\beta Z^{\alpha\beta}\qquad\text{for each } i,j \in \{1,\ldots,d\}.
\end{align*}

One further notation: given tensors $\Upsilon, \Theta \in \mathfrak{gl}(d;\C)^{\otimes n}$ with arbitrary $d,n \in \mathbb{N}_{\geq 1}$, we define a natural product, denoted as $\Upsilon\Theta\in \mathfrak{gl}(d;\C)^{\otimes n}$, which reduces to  matrix multiplication when  $n=1$. In Example~(Diagonalisable $\scrm$) we shall to compute $Q^{\otimes n}\Lambda Q^{\otimes (-n)}$ for $Q\in\mathfrak{gl}(d;\C)$,  $\Lambda\in\mathfrak{gl}(d;\C)^{\otimes n}$.

For the sake of clarity, we first consider the case $n=2$. Under the identification $\mathfrak{gl}(d;\C) \cong \left(\C^d\right)^* \otimes \C^d$, we can represent each matrix $A \in \mathfrak{gl}(d;\C)$ componentwise as $A = \left\{A^i_j\right\}$. Then, for $A,B,C,D \in \mathfrak{gl}(d;\C)$ and $i,j,k,\ell \in \{1,\ldots,d\}$ we have
\begin{align*}
    \big[(AB) \otimes (CD)\big]^{ij}_{k\ell} = \sum_{\alpha,\beta=1}^d A^i_\alpha B^\alpha_k C^j_\beta D^\beta_\ell = \big[ (A\otimes C) (B\otimes D) \big]^{ij}_{k\ell}.
\end{align*}
The second equality should be understood as follows: for $\xi, \theta \in \mathfrak{gl}(d;\C)^{\otimes 2} \cong \left(\C^d\right)^* \otimes \C^d \otimes \left(\C^d\right)^* \otimes \C^d$, write $\xi = \left\{\xi^{ij}_{k\ell}\right\}$ and $\theta = \left\{\theta^{ij}_{k\ell}\right\}$. Then the juxtaposition $\xi\theta \in \mathfrak{gl}(d;\C)^{\otimes 2}$ is given by
\begin{align*}
   \big[\xi \theta\big]^{ij}_{k\ell} = \sum_{\alpha,\beta=1}^d\xi^{ij}_{\alpha\beta} \theta^{\alpha\beta}_{k \ell}. 
\end{align*}
When $n=1$ this is just the usual matrix multiplication. In particular, when $\xi = A \otimes C$ for $A,C\in \mathfrak{gl}(d;\C) \cong \left(\C^d\right)^* \otimes \C^d$, one takes  $\left(A\otimes C\right)^{ij}_{\alpha\beta} = A^i_\alpha C^j_\beta$ for any $i,j,\alpha,\beta \in \{1,\ldots,d\}$.

For general $n \in \mathbb{N}$, as $\dim \left[\mathfrak{gl}(d;\C)^{\otimes n}\right] = d^{2n}$, there are $2n$ indices for each generic element in $\mathfrak{gl}(d;\C)^{\otimes n}$. Thus, let us consider an arbitrary pair of multi-indices $I=(i_1, \ldots, i_n)$ and $J = (j_1, \ldots, j_n) \in \{1,\ldots,d\}^n$. Each $\Upsilon \in \mathfrak{gl}(d;\C)^{\otimes n}$ can be represented componentwise as
\begin{align*}
    \Upsilon = \left\{\Upsilon^I_J\right\} := \left\{\Upsilon^{i_1, \ldots, i_n}_{j_1, \ldots, j_n}\right\}.
\end{align*} 
Then, for $\Upsilon, \Theta \in \mathfrak{gl}(d;\C)^{\otimes n}$ the product $\Upsilon \Theta = \left\{\Upsilon \Theta^I_J \right\}  \in \mathfrak{gl}(d;\C)^{\otimes n}$ is given componentwise by
\begin{align*}
\left[\Upsilon \Theta\right]^I_J = \sum_{K\in \{1,\ldots,d\}^n} \Upsilon^I_K \Theta^K_J = \sum_{k_1=1}^d \cdots \sum_{k_n=1}^d \Upsilon^{i_1, \ldots, i_n}_{k_1, \ldots, k_n} \Upsilon^{k_1, \ldots, k_n}_{j_1, \ldots, j_n}.
\end{align*}


\subsubsection{Expectation of tensor powers of multi-dimensional Gaussians}

In the proof of Theorem~\ref{main theorem for singular limit}, a crucial step is to evaluate $\E^p \left[\left(P_{t}\right)^{\otimes {n}}\right]$. The momentum $P_t$ is a $d$-dimensional Gaussian vector: 
\begin{lemma}\label{lem: mean and cov}
$\left\{P_t\right\}$ is normally distributed with mean and covariance 
\begin{align*}
\boldsymbol \mu_t = e^{-\mom t} p; \qquad  \boldsymbol \Sigma_t = \int_0^t e^{-\mom s}e^{-\frac{\scrm^*}{m}s} \,\dd s\qquad\text{(independent of $p$)}.
\end{align*}
\end{lemma}


\begin{lemma}\label{lem:GaussianMoment}
Let $G\sim \mathcal{N}(\boldsymbol{\mu}, \mathbf{\Sigma})$ be a $d$-dimensional Gaussian variable. Then its $q^{\text{th}}$ moment for any $q \in \mathbb{N}$ is given by  $\mathcal{M}^q :=  \E\left[G^{\otimes q}\right] = \sum_{k=0}^{\lfloor q/2\rfloor}{q \choose 2k}\frac{(2k)!}{k!2^k}\,\sym\left(\boldsymbol{\mu}^{\otimes (q-2k)}\otimes \mathbf{\Sigma}^{\otimes k}\right).$
\end{lemma}

See \cite[Theorem 4.1]{Pereira2022} for a proof of Lemma~\ref{lem:GaussianMoment}. The first four moments are thus $ \mathcal{M}^1 = \boldsymbol{\mu}$, $\mathcal{M}^2 = \boldsymbol{\mu}^{\otimes 2}+\boldsymbol{\Sigma}$, $\mathcal{M}^3  = \boldsymbol{\mu}^{\otimes 3}+3\,\sym(\boldsymbol{\mu}\otimes \mathbf{\Sigma})$, and $\mathcal{M}^4 =\boldsymbol{\mu}^{\otimes 4}+6\,\sym\left(\boldsymbol{\mu}^{\otimes 2}\otimes \mathbf{\Sigma}\right)+3\,\sym\left(\mathbf{\Sigma}^{\otimes 2}\right)$. Here and hereafter, for a rank-$k$ tensor $T$ we set $\sym(T) := \frac{1}{k!} \sum_{\sigma} \sum_{i_1,\ldots, i_k} T_{i_1, \ldots, i_k} e_{i_{\sigma(1)}} \otimes \ldots\otimes e_{i_{\sigma(k)}}$, where $\sigma$ ranges over permutations of $\{1,\ldots,k\}$ and  $i_j$ ranges over $\{1,\ldots,d\}$.  In the sequel we shall use Lemma~\ref{lem:GaussianMoment} to compute $\E^p \left[\left(P_{t}\right)^{\otimes {n}}\right]$.

\subsubsection{A trick of symmetrisation}


\begin{lemma}\label{lem: sym}
Let $f \in \C[t]$ and $L>0$. Denote the simplex $$\blacktriangle \equiv \blacktriangle_{n,L} := \left\{ \left(t_1, \ldots, t_n\right) \in \R^n:\, 0 \leq t_1<t_2< \ldots < t_n \leq L\right\}.$$ Let $F$ be the symmetric polynomial given by $F(t_1, \ldots, t_n) = \sym \big( f(t_1) \otimes \ldots \otimes f(t_n)\big)$. Then 
\begin{align*}
    \int_{\blacktriangle} F\left(t_1, \ldots, t_n\right)\,\dd t_1 \ldots\,\dd t_n = \frac{1}{n!} \left[\int_0^L f(t)\,\dd t\right]^{\otimes n}.
\end{align*}
\end{lemma}



\subsection{First two levels of the expected signature  $\Phi^\mm$ in the massless limit}\label{subsec: Phi 1--4} 


\subsubsection{$\Phi_1^\mm$}\label{subsubsec, Phi1}

Recall $\Phi_0 = (1,0,0,\ldots)$. The source term in Eq.~\eqref{PDE for Phi, later} is thus $\fs_1^\mm(t,p) = -\mom p$. Hence, by the Feynman--Kac formula~\eqref{feynmann-kac, later}, one has
\begin{align}\label{Phi 1 computation}
\Phi_1^\mm(t,p) = -\mom \int_0^t e^{-\mom s}p\,\dd s =\left(e^{-\mom t}-\id\right)p.
\end{align}

\subsubsection{$\Phi_2^\mm$} \label{subsubsec, Phi2}
Recall that for some $0<\lambda \leq \Lambda <\infty$ and $0<K<\infty$, one has $\|\scrm\|\leq \Lambda$ and  $\max\left\{\left\|e^{-\scrm^*\delta}\right\|,\,\left\|e^{-\scrm\delta}\right\|\right\} \leq K e^{-\lambda\delta}$ for each $\delta>0$. The existence of $\lambda$ follows from the positive definiteness of $\scrm$; see Lemma~\ref{lem: matrix expoential limit}. By Eqs.~\eqref{PDE for Phi, later} $\&$ \eqref{feynmann-kac, later} (see \S\ref{subsec: Phi n polyl in p} for notations) we have
\begin{align*}
    \Phi_2^\mm(t,p) =
    \Phi_2^{\mm,\,{\rm high}}(t,p) + \Phi_2^{\mm,\,{\rm low}}(t,p):= \E^p\left[ \int_0^t \fs^{\rm high}_2(t-s, P_s)\,\dd s \right] + \E^p\left[ \int_0^t \fs^{\rm low}_2(t-s, P_s)\,\dd s \right].
\end{align*} 
For the degree-2 (in $p$) term, we use Fubini's theorem to take the expectation $\E^p$ under the  integrals with respect to time parameters, thus obtaining
\begin{align*}
  \Phi_2^{\mm,\,{\rm high}}(t,p) &= \int_0^t \int_0^{t-s} \left(-\mom\right) \otimes \left(-\mom e^{-\mom r}\right)\E^p\left[P_s \otimes P_s\right]\,\dd r\,\dd s.
\end{align*}
Recall from Lemma~\ref{lem:GaussianMoment} that $\E^p\left[P_s \otimes P_s\right] = \boldsymbol{\mu}_s^{\otimes 2} + \boldsymbol{\Sigma}_s$  (see Lemma~\ref{lem: mean and cov} for $\boldsymbol{\mu}_s^{\otimes 2}$, ${\bf \Sigma}_s$). Then
\begin{align*}
\Phi_2^{\mm,\,{\rm high}}(t,p) &= \one_1(t,p) + \one_2(t) := \int_0^t \int_0^{t-s} \left(-\mom\right) \otimes \left(-\mom e^{-\mom r}\right) \cdot \left(e^{-\mom s}p \otimes e^{-\mom s}p\right)\,\dd r\,\dd s\\
&\quad + \int_0^t \int_0^{t-s}\left(-\mom\right) \otimes \left(-\mom e^{-\mom r}\right) \cdot \left\{\int_0^s e^{-\mom \varsigma} e^{-\frac{\scrm^*}{m}\varsigma}\,\dd \varsigma
\right\}\,\dd r\,\dd s.
\end{align*}

For $\one_1$, consider the change of variables $t_1 := \frac{s}{m}$, $t_2 := \frac{r+s}{m}$, which leads to 
\begin{align*}
\one_1(t,p) &= \iint_{0 \leq t_1 < t_2 \leq \frac{t}{m}} \left\{ \left(-\scrm e^{-\scrm t_1}p\right) \otimes \left(-\scrm e^{-\scrm t_2}p\right)
    \right\}\,\dd t_1 \,\dd t_2=: \fraka_2^\mm(t,p).
\end{align*}

For $\one_2$, we compute by the mixed product property for tensor products (see \S\ref{Tensor product of linear mappings} for  notations), a change of variables $\sigma := \frac{s}{m}$, and Fubini's theorem that 
\begin{align*}
\one_2(t) &= \int_0^t  \left\{(-\scrm) \otimes\int_0^{t-s} \left( -\mom e^{-\mom r}    \right)\,\dd r\right\} \left\{ \int_0^{s/m} e^{-\scrm\varsigma} e^{-\scrm^* \varsigma}\,\dd \varsigma\right\}\,\dd s\\
& = m \int_0^{\frac{t}{m}} \left(-\scrm\right) \otimes \left(e^{-\scrm \left(\frac{t}{m}- \sigma\right)}-\id\right) \cdot \ssig_\sigma \,\dd\sigma\\
& = t \dashint_0^{\frac{t}{m}} \left(\scrm \otimes \id\right) \cdot \ssig_\sigma\,\dd \sigma  - m \int_0^{\frac{t}{m}} \left\{\scrm \otimes e^{-\scrm \left(\frac{t}{m} - \sigma \right)}\right\}\cdot\ssig_\sigma\,\dd \sigma  =: \one_{2, 1}^\mm (t) + \one_{2,2}^\mm (t).
\end{align*}
Observe the naive bound:
\begin{align}\label{Sigma estimate, simple}
\left\|\ssig_T\right\| = \left\|\int_0^T e^{-\scrm t}e^{-\scrm^* t}\,\dd t\right\| \leq \frac{K^2d(1 - e^{-2\lambda T})}{2 \lambda} \leq \frac{K^2d}{2\lambda} \quad \text{ for any } T>0.
\end{align}
Thus $\one_{2,2}^\mm(t) \lesssim \mathcal{O}(m)$ in the small mass limit:
\begin{align*}
\left\|\one_{2,2}^\mm(t)\right\| &\leq m \int_0^{\frac{t}{m}} \left\| \scrm \otimes e^{-\scrm \sigma'} 
 \cdot \ssig_{\frac{t}{m}- \sigma'}\right\|\,\dd \sigma'\leq \frac{m \Lambda K^3d^3}{2\lambda} \int_0^{\frac{t}{m}} e^{-\lambda\sigma'}\,\dd \sigma' \leq \frac{\Lambda K^3d^3}{2\lambda^2} \,m.
\end{align*}

Now let us turn to 
\begin{align*}
\Phi_2^{\mm,\,{\rm low}}(t,p) = \E^p\left[ \int_0^t \fs^{\rm low}_2(t-s, P_s)\,\dd s \right]= \sum_{j=1}^d \E^p\left[\int_0^t \left\{e_j \otimes \p_{p_j}\Phi_1^\mm(t-s, P_s) + \frac{1}{2}e_j \otimes e_j \right\}\,\dd s \right].
\end{align*}
In view of the computation for $\Phi_1$ in Eq.~\eqref{Phi 1 computation}, we note by a change of variables that $ \Phi_2^{\mm,\,{\rm low}}(t,p) = -\frac{\id}{2}t\,  + m \int_0^{\frac{t}{m}} e^{-\scrm \sigma }\,\dd \sigma$. Thus $\left\|\Phi_2^{\mm, {\rm low}} (t,p) + \frac{\id}{2}t \right\| \leq \frac{mK}{\lambda}$. 

Therefore, we may deduce from the above arguments 
\begin{align}\label{Phi 2, final bound}
\Phi_2^\mm(t,p) &= \fraka_2^\mm(t,p) +  \frakb_2^\mm(t) + \err_2^\mm(t), \qquad \frakb_2^\mm(t) := \left(\scrm \cdot \dashint_0^{\frac{t}{m}} \ssig_\sigma\,\dd \sigma - \frac{\id}{2}\right) t, 
\end{align}
with the error term $    \err_2^\mm(t) := m \int_0^{\frac{t}{m}} e^{-\scrm \sigma }\,\dd \sigma - m \int_0^{\frac{t}{m}} \left\{\scrm \otimes e^{-\scrm \left(\frac{t}{m} - \sigma \right)}\right\}\cdot\ssig_\sigma\,\dd \sigma$ satisfying
\begin{align}\label{Phi2, error bound}
\left\|\err_2^\mm(t)\right\| \leq m \left(\frac{K}{\lambda} + \frac{\Lambda K^3d^3}{2\lambda^2}\right).
\end{align}
This verifies Theorem~\ref{main theorem for singular limit} for the case $n=2$.

It is crucial to observe that 
both $\frakb_2^\mm(t)$ and $\err_2^\mm(t)$ are independent of $p$. Also,
\begin{align*}
\lim_{m \to 0^+}\frakb_2^\mm (t) = \left(\scrm \bra\ssig\ket - \frac{\id}{2}\right) t \quad\text{ in the $\ell^\infty$-tensor norm},
\end{align*}
where for tensor-valued  function $f$ one designates $    \bra{f}\ket:= \lim_{T \to \infty} \frac{1}{T} \int_0^T f(t)\,\dd t.$ The arguments for Corollary~\ref{cor, explicit limit} show that the long-time average $\bra\ssig\ket$ equals $\scrc \equiv \ssig_\infty = \int_0^\infty e^{-\scrm t}e^{-\scrm^*t}\,\dd t$.

Things get rapidly involved from $n=3$. Computations of $\Phi_3^\mm$ are presented in the (online) appendix, which contains  essential ingredients of the proof of Theorem~\ref{main theorem for singular limit} in \S\ref{subsec: limit, Dec24}.

\subsection{Small mass limit at arbitrary degree of tensors}\label{subsec: limit, Dec24}

\begin{proof}[Proof of Theorem~\ref{main theorem for singular limit}]  
{We induct on the signature level $n$. The induction hypothesis is that 
for any $n \geq 2$, the $n^{\text{th}}$-level expected signature can be expressed as
\begin{align}\label{eqn: esig_dec}
    \Phi_n^\mm(t,p) = \fraka_n^\mm(t,p) + \frakb_n^\mm(t,p) + \err_n^\mm(t,p),
\end{align}
where
\begin{eqnarray}\label{ind hypo, xx}
&&\fraka_n^\mm(t,p) := \idotsint\limits_{0\leq t_1<\dots<t_n\leq t} \,\,\bigotimes_{j=1}^n  \left( -\mom e^{-\mom t_j} p\right)\, \dd t_1\ldots\,\dd t_n,\\
&&\frakb_n^\mm(t,p) := \sum_{k=1}^{\lfloor {n}\slash{2}\rfloor} \left\{\fraka_{n-2k}^\mm(t,p) \otimes \Bigg[\left( \scrm\cdot\dashint_0^{{t}\slash{m}}\ssig_\sigma \,\dd \sigma-\frac{\id}{2}\right) 
t\Bigg]^{\otimes k}  \right\}, \\
&&\left|\err_n^\mm(p,t)\right| \leq C(K,\Lambda, \lambda^{-1}, n,d,t)\left(1+|p|^{n-2}\right)\cdot m. 
\end{eqnarray}
Here $\err_k^\mm(p,t)$ is a polynomial of $p$ with degree $\leq (n-2)$. In particular, $\left\|\err_n^\mm \equiv {\rm Error}_n^\mm\right\| \lesssim \mathcal{O}(m)$ when $m \to 0^+$, as asserted by the theorem. The cases $n=1, 2$ have already been established in \S\S\ref{subsubsec, Phi1}  $\&$ \ref{subsubsec, Phi2}. Now we assume for $k \in\{1,\ldots, n-1\}$ and argue for $k=n$.
}

For notational convenience, in this proof we write for each $t>0$ that
\begin{equation}\label{Xi, def}
    \Xi^\mm(t):= \scrm \cdot\dashint_0^{t/m} \ssig_\sigma\,\dd\sigma - \frac{\id}{2} \in \mathfrak{gl}(d;\R).
\end{equation}
Thus,\footnote{For any matrix $\mathfrak{M}\in \mathfrak{gl}(d;\R)$ we adhere to the convention $\mathfrak{M}^{\otimes 0} \equiv 1$. } 
 $\fraka_n^\mm (t,p)= \fraka_{n-2k}^\mm(t,p) \otimes \left[\Xi^\mm(t)\right]^{\otimes k}\bigg|_{k=0}.$ Also, by the estimate in Eq.~\eqref{Sigma estimate, simple},
\begin{align}\label{Xi m, uniform bd}
\sup_{T > 0}\left\|\Xi^\mm(T)\right\| \leq \frac{\Lambda K^2d^2}{2\lambda} + \frac{1}{2}.
\end{align}
Hence, by Lemma~\ref{lem: average difference} and Eq.~\eqref{Xi m, uniform bd}, 
\begin{align}\label{Xi estimate}
\left\|\Xi^\mm(t)-\Xi^\mm(t-m\sigma)\right\| \leq  \left(\frac{\Lambda K^2d^2}{2\lambda}+\frac{1}{2}\right)\frac{2m\sigma}{t} \leq \frac{2\Lambda K^2d^2}{\lambda}\cdot \frac{m\sigma}{t}.
\end{align}



In provision of the induction hypothesis (Eq.~\eqref{eqn: esig_dec}) on degree $n-1$ and $n-2$, the Feynman--Kac formula~\eqref{feynmann-kac, later}, and the expression~\eqref{PDE for Phi, later}, we may decompose
\begin{align}\label{nine terms}
    \Phi_n^\mm &= \sum_{j=1}^9\two_{n, j}^\mm,
\end{align}
with 
\begin{footnotesize}
\begin{eqnarray*}
    && \two_{n, 1}^\mm(t,p) = \E^p\left[-\int_0^t \mom P_s \otimes \fraka_{n-1}^\mm(t-s, P_s) \,\dd s\right],\two_{n, 2}^\mm(t,p) = \E^p\left[-\int_0^t \mom P_s \otimes \frakb_{n-1}^\mm(t-s, P_s) \,\dd s\right],\\
    &&\two_{n, 3}^\mm(t,p) = \E^p\left[-\int_0^t \mom P_s \otimes \err_{n-1}^\mm(t-s, P_s) \,\dd s\right],
    \two_{n, 4}^\mm(t,p) = \sum_{j=1}^d e_j \otimes   \E^p\left[\int_0^t \left(\p_{p_j}\fraka_{n-1}^\mm\right)(t-s, P_s) \,\dd s\right],\\
    &&\two_{n, 5}^\mm(t,p) = \sum_{j=1}^d e_j \otimes   \E^p\left[\int_0^t  \left(\p_{p_j}\frakb_{n-1}^\mm\right)(t-s, P_s) \,\dd s\right],\two_{n, 6}^\mm(t,p) = \sum_{j=1}^d e_j \otimes   \E^p\left[\int_0^t\left(\p_{p_j}\err_{n-1}^\mm\right)(t-s, P_s) \,\dd s\right],\\
    &&\two_{n, 7}^\mm(t,p) = \frac{1}{2}\sum_{j=1}^d e_j \otimes e_j \otimes   \E^p\left[\int_0^t \fraka_{n-2}^\mm(t-s, P_s) \,\dd s\right],\two_{n, 8}^\mm(t,p) = \frac{1}{2}\sum_{j=1}^d e_j \otimes e_j \otimes   \E^p\left[\int_0^t \frakb_{n-2}^\mm(t-s, P_s) \,\dd s\right],\\
    &&\two_{n, 9}^\mm(t,p) = \frac{1}{2}\sum_{j=1}^d e_j \otimes e_j \otimes   \E^p\left[\int_0^t \err_{n-2}^\mm(t-s,P_s) \,\dd s\right].
\end{eqnarray*}
\end{footnotesize}

The remaining parts of the proof are devoted to estimating these nine terms. For each term we shall further decompose 
${\rm II}_{n,j}^{(m)}(t,p) = {\rm II}_{n,j, {\rm good}}^\mm(t,p) + {\rm Error}_{n, j}^\mm(t,p), $
where ${\rm II}_{n,j, {\rm good}}^\mm(t,p)$ is a polynomial in $p$ of degree $\leq {n-2}$ and, for each $1\leq j \leq 9$,
\begin{align}\label{error, final}
    \left\| {\rm Error}_{n,j}^\mm(t,p)\right\| \leq C\left(K, \Lambda, \lambda^{-1}, n,d,t \right)\cdot \left(1+|p|^{n-2}\right)\cdot m.
\end{align}
We then show that $\Phi_{n, {\rm good}}^\mm(t,p) \equiv \sum_{j=1}^{9} {\rm II}_{n,j, {\rm good}}^\mm(t,p)$ takes the form 
\begin{align*}
&\Phi_{n, {\rm good}}^\mm(t,p)  = \fraka_n^\mm(t,p) + \frakb_{n}^\mm(t,p)  \qquad \text{if $n$ is odd},
\end{align*}
while 
\begin{footnotesize}
\begin{align*}
&\Phi_{n, {\rm good}}^\mm(t,p) = \fraka_n^\mm(t,p) +  \frakb_{n}^\mm(t,p) - \left[\Xi^\mm(t)\right]^{\otimes \frac{n}{2}} t^{\frac{n}{2}}\\
&\quad +\sum_{j=1}^d e_j \otimes \Bigg[ \int_0^t \left[{-}\id {+} e^{-\mom (t-s)}\right]^{j}_{\bullet}\otimes \left[\Xi^\mm(t-s)\right]^{\otimes \left(\frac{n}{2}-1\right)} (t-s)^{\frac{n}{2}-1}\,\dd s\Bigg] \\
&\quad + \frac{\id}{2}\otimes  \left[\int_0^t \left[\Xi^\mm(t-s)\right]^{\otimes  \left(\frac{n}{2}-1\right)}(t-s)^{\left(\frac{n}{2}-1\right)} \,\dd s\right] + \scrm \cdot \left(\dashint_{0}^{\frac{t}{m}} \ssig_\sigma \,\dd \sigma \right) \otimes \left[\Xi^\mm(t)\right]^{\otimes \left(\frac{n}{2}-1\right)} t^{\frac{n}{2}}   \quad \text{if $n$ is even}.
\end{align*}
    
\end{footnotesize}

In fact, nontrivial contributions to $\Phi_{n, {\rm good}}^\mm(t,p)$ arise only from $\two_{n,j}^\mm$ with $j \in \{1,2,5,8\}$. 

\subsubsection{Estimate for $\two_{n, 1}^\mm$}\label{subsubsec: II n,1} 

This is the most technical term among all ${\rm II}^\mm_{n,j}$. We present the complete details of it to avoid subsequent repetitions, as the arguments for the other (easier) terms all have similar ingredients. We show that $\two_{n, 1, good}^\mm = \fraka_n^\mm(t,p)$ and ${\rm Error}_{n, j}^\mm(t,p)= \ir_{n, {\rm REG}}^\mm(t,p) + \ir_{n, {\rm SING}}^\mm(t,p)$ is controlled by $C\left(K,\Lambda, \lambda^{-1},n,t\right) \left(1+|p|^{n-2}\right) \cdot m$ for $n > 2$.

First, from the expression for $\fraka^\mm_{n-1}$ and Fubini's theorem, we deduce that
\begin{footnotesize}
\begin{align*}
&\two_{n, 1}^\mm (t,p) =\E^p\Bigg[\int_0^t -\mom P_s \otimes \Bigg\{\,\,\,\idotsint\limits_{0\leq t_1<\cdots<t_{n-1}\leq t-s}\left(-\mom e^{-\mom t_1}P_s\right)\otimes  \cdots\otimes\left(-\mom e^{-\mom t_{n-1}}P_s\right)\,\dd t_1\dots  \dd t_{n-1}\Bigg\}\,\dd s \Bigg] \nonumber\\
&=\int_0^t\idotsint\limits_{0\leq t_1<\cdots<t_{n-1}\leq t-s} \left[\left(-\mom\right) \otimes \left(-\mom e^{-\mom t_1} \right) \otimes \cdots \otimes\left( -\mom e^{-\mom t_{n-1}} \right)\right] \cdot \left\{\E^p \left[P_s^{\otimes {n}}\right]\right\} \,\dd t_1\dots \dd t_{n-1} \,\dd s.
\end{align*}
\end{footnotesize}
A change of variables $\tau_j := t_j+s$ for $j=1,2,\ldots,n$ and $\tau_0 := s$ yields\footnote{For $d\times d$ matrices $A_1,\ldots, A_n$ and $B \in  \left(\R^d\right)^{\otimes n}$, $\left(A_1 \otimes \ldots \otimes A_n\right)\cdot B$ is a $d\times n$ matrix --- $
\left[\left(\bigotimes_{j=1}^n A_j\right) \cdot B \right]^{i_1, \ldots i_n} = \sum_{j_1=1}^{d} \cdots \sum_{j_n=1}^{d}
\left(A_1\right)^{i_1}_{j_1}\left(A_2\right)^{i_2}_{j_2} \cdots \left(A_n\right)^{i_n}_{j_n} B^{j_1, \ldots, j_n}$ for any $i_1, \ldots, i_n \in \left\{1,\ldots,d\right\}$. See \S\ref{Tensor product of linear mappings}.}
\begin{align}\label{new, x}
\two_{n, 1}^\mm (t,p) &=\idotsint\limits_{0\leq\tau_0 <\cdots<\tau_{n-1}\leq t} \Bigg[\left(-\mom\right) \otimes \left(-\mom e^{-\mom (\tau_1-\tau_0)} \right) \otimes\cdots  \nonumber\\
&\qquad \otimes \left(-\mom e^{-\mom (\tau_{n-1}-\tau_0)} \right)\Bigg]\cdot \E^p \left[\left(P_{\tau_0}\right)^{\otimes {n}}\right] \,\dd \tau_0 \cdots \dd \tau_{n-1}.
\end{align}

By virtue of Lemma~\ref{lem:GaussianMoment},
\begin{align}\label{new, P tens n }
    \E^p \left[\left(P_{\tau_0}\right)^{\otimes {n}}\right] = \left( e^{-\frac{M}{m}\tau_0} p \right)^{\otimes n} + \mathcal{R}^\mm_{n}(p,\tau_0).
\end{align}
The remainder consists of mixed terms of mean and covariance:
\begin{align*}
\mathcal{R}_{n}^\mm(\tau_0,p) &= \sum_{k=1}^{\lfloor n\slash 2 \rfloor} \frac{(2k)!}{k! 2^k} \sym \left\{ \left(e^{-\mom\tau_0}p\right)^{\otimes (n-2k)} \otimes \left(\int_0^{\tau_0} e^{-\mom s}e^{-\frac{\scrm^*}{m}s} \,\dd s\right)^{\otimes {k}} \right\}.
\end{align*}
The key point here is that the summation starts from $k=1$. As the covariance matrix does not depend on $p$, we observe that $\deg_p \mathcal{R}_{n}^\mm \leq n-2$. 

Now,  Eqs.~\eqref{new, x} $\&$ \eqref{new, P tens n } yield that
\begin{align}\label{II 1, decomp}
&\two_{n,1}^\mm(t,p) =\idotsint\limits_{0\leq\tau_0 <\cdots<\tau_{n-1}\leq t} \bigg[\left(-\mom e^{-\mom\tau_0}\right) \otimes \left(-\mom e^{-\mom \tau_1 } \right) \otimes \cdots \nonumber\\
& \otimes \left( -\mom e^{-\mom \tau_{n-1}} \right)\bigg] \cdot\left ( p^{\otimes n}\right) \,\dd \tau_0\ldots \dd \tau_{n-1} + \ir_{n}^\mm(t,p)   =: \fraka_n^\mm(t,p) + \ir_{n}^\mm(t,p),
\end{align} 
where the remainder $\ir_{n}^\mm$ is an $n$-fold iterated integral of $\mathcal{R}_{n}^\mm$ in Eq.~\eqref{new, P tens n }:
\begin{align}\label{IR def}
\ir_{n}^\mm (t,p) &:= \idotsint\limits_{0\leq\tau_0 <\cdots<\tau_{n-1}\leq t}  \bigg[\left(-\mom\right) \otimes \left(-\mom e^{-\mom (\tau_1-\tau_0)} \right) \otimes \nonumber \\
&\qquad \cdots\otimes \left( -\mom  e^{-\mom (\tau_{n-1}-\tau_0)} \right)\bigg] \cdot  \mathcal{R}^\mm_{n}(p,\tau_0) \,\dd \tau_0\dots \dd \tau_{n-1}. 
\end{align}

To proceed, we further split the source term $\mathcal{R}_{n}^\mm$ into
\begin{align*}
&\mathcal{R}_{n}^\mm(\tau_0,p) = \mathcal{R}_{n, {\rm REG}}^\mm(\tau_0,p) 
 + \mathcal{R}_{n,{\rm SING}}^\mm(\tau_0,p)\\
&\quad:=\sum_{1 \leq k < \lfloor n\slash 2 \rfloor} \frac{(2k)!}{k! 2^k} \sym \left\{ \left(e^{-\mom\tau_0}p\right)^{\otimes (n-2k)} \otimes \left(\int_0^{\tau_0} e^{-\mom s}e^{-\frac{\scrm^*}{m}s} \,\dd s\right)^{\otimes {k}} \right\} \\ 
&\qquad + \sum_k \delta_{k, \frac{n}{2}} \frac{(2k)!}{k! 2^k} \sym \left\{ \left(e^{-\mom\tau_0}p\right)^{\otimes (n-2k)} \otimes \left(\int_0^{\tau_0} e^{-\mom s}e^{-\frac{\scrm^*}{m}s} \,\dd s\right)^{\otimes {k}} \right\}.
\end{align*}
Note that $\mathcal{R}_{n,{\rm SING}}^\mm(p,\tau_0) \equiv 0$ unless $n$ is even. As per Eq.~\eqref{IR def}, we set for $\bullet \in \left\{{\rm REG},\,{\rm SING}\right\}$:
\begin{footnotesize}
\begin{align*}    \ir_{n,\bullet}^\mm(t,p) &:= \idotsint\limits_{0\leq\tau_0 <\cdots<\tau_{n-1}\leq t}  \bigg[\left(-\mom\right) \otimes \left(-\mom e^{-\mom (\tau_1-\tau_0)} \right) \otimes \cdots\otimes \left( -\mom  e^{-\mom (\tau_{n-1}-\tau_0)} \right)\bigg] \cdot  \mathcal{R}^\mm_{n,\bullet}(p,\tau_0) \,\dd \tau_0\dots \dd \tau_{n-1}. 
\end{align*}
\end{footnotesize}

\noindent
\underline{Estimate for $\ir^\mm_{n,{\rm REG}}$.}  Setting $t_j := \tau_j \slash m$ for $j \in \{0,1,\ldots, \tau_{n-1}\}$, we get
\begin{align*}
\ir^\mm_{n,{\rm REG}} &= \idotsint\limits_{0\leq t_0 <\cdots<t_{n-1}\leq \frac{t}{m}} \left[ \bigotimes_{j=0}^{n-1} \left(-\scrm e^{-\scrm (t_j-t_0)} \right) \right]  \cdot \mathcal{R}^\mm_{n,{\rm REG}}(mt_0,p)\,\dd 
 t_0\ldots\,\dd t_{n-1}.
\end{align*}
Singling out the integral over $t_0$ and considering $\delta_j := t_j-t_0$, we  deduce that
\begin{align*}
\ir^\mm_{n,{\rm REG}} &= \int_0^{\frac{t}{m}}\overbrace{\left\{\,\,\,\idotsint\limits_{0\leq \delta_1 <\cdots<\delta_{n-1}\leq \frac{t}{m}-t_0} \left[ \bigotimes_{j=0}^{n-1} \left(-\scrm e^{-\scrm \delta_j} \right) \right]  \,\dd\delta_1\ldots\,\dd\delta_{n-1}\right\}}^{=:\,\mathcal{J}_{n-1}^\mm(t_0)} \cdot \mathcal{R}^\mm_{n,{\rm REG}}(mt_0, p) \,\dd t_0.
\end{align*}

Now, notice that 
\begin{align*}
\left\|\mathcal{J}_{n-1}^\mm(t_0)\right\| &\leq \idotsint\limits_{0\leq \delta_1 <\cdots<\delta_{n-1}\leq \frac{t}{m}-t_0} \left\|\scrm\right\| \left\| \bigotimes_{j=1}^{n-1} \left(-\scrm e^{-\scrm \delta_j} \right) \right\|\,\dd\delta_1\ldots\,\dd\delta_{n-1}\\
&\leq \Lambda^{n}K^{n-1} \left\{\,\,\,\idotsint\limits_{0\leq \delta_1 <\cdots<\delta_{n-1}\leq \frac{t}{m}-t_0} \left(\prod_{j=1}^{n-1} e^{-\lambda\delta_j}\right)\,\dd\delta_1\ldots\,\dd\delta_{n-1}\right\}\\
&= \frac{\Lambda^{n}K^{n-1} e^{-\lambda}}{(n-1)! \lambda^{n-1}} \left[ 1- e^{-\lambda\left(\frac{t}{m} - t_0\right)} \right]^{n-1}\leq \frac{\Lambda^{n}K^{n-1}e^{-\lambda}}{(n-1)! \lambda^{n-1}},
\end{align*}
which is uniform in $m$ and $t_0$. The second line follows from the assumptions that $\|\scrm\|\leq \Lambda$ and $\left\|e^{-\scrm \delta}\right\| \leq K e^{-\lambda\delta}$, and the third line from Lemma~\ref{lem: sym}.

On the other hand, for $\mathcal{R}_{n, {\rm REG}}^\mm$, by a change of variables we have
\begin{align*}
\mathcal{R}_{n, {\rm REG}}^\mm(\tau_0,p) 
=\sum_{1 \leq k < \lfloor n\slash 2 \rfloor} \frac{(2k)!}{k! 2^k} \left( m^k\right) \sym \left\{ \left(e^{-\mom\tau_0}p\right)^{\otimes (n-2k)} \otimes  {\ssig_{\tau_0 \slash m}}^{\otimes {k}} \right\}, 
\end{align*}
where $\ssig_t := \int_0^t e^{-\scrm s}e^{-\scrm^* s}\,\dd s$. The previous estimate for $\left\|\mathcal{J}_{n-1}^\mm(t_0)\right\|$ yields
\begin{align}\label{IR REG bound, final}
&\left\|\ir^\mm_{n,{\rm REG}}(t,p)\right\| \leq \int_0^{t \slash m}d^n \left\|\mathcal{J}^\mm_{n-1}(t_0) \right\|\left\|\mathcal{R}^\mm_{n,{\rm REG}}(mt_0,p)\right\|\,\dd t_0\nonumber\\
&\qquad \leq \frac{\Lambda^n K^{n-1}d^n e^{-\lambda}}{(n-1)! \lambda^{n-1}} \int_0^{t \slash m}\left\{\sum_{1 \leq k < \lfloor n\slash 2 \rfloor}  \frac{(2k)!}{k! 2^k} m^k \left\|{\bf \Sigma}_{t_0}\right\|^k  \|e^{-\frac{\scrm}{m}t_0}p\|^{n-2k} \right\}\,\dd t_0\nonumber\\
&\qquad \leq  \frac{\Lambda^nK^{n-1}d^n e^{-\lambda}}{(n-1)! \lambda^{n-1}} \sum_{1 \leq k < \lfloor n\slash 2 \rfloor} \Bigg\{ \frac{(2k)!}{k! 2^k} m^k \left(\frac{K^2d}{2\lambda}\right)^k\times\int_0^{t \slash m}\left(Kd|p|e^{-\lambda t_0}\right)^{n-2k} \,\dd t_0 \Bigg\}\nonumber\\
&\qquad \leq  \frac{\Lambda^nK^{2n} d^{2n} e^{-\lambda}}{(n-1)! \lambda^{n}} \sum_{1 \leq k < \lfloor n\slash 2 \rfloor} \frac{(2k)!|p|^{n-2k}}{k! 2^{2k} \lambda^k(n-2k)}\cdot m^k.
\end{align}
In the above, the third line holds by $\left\|{\bf \Sigma}_{t_0}\right\| \leq \int_0^{t_0}K^2d e^{-2\lambda s}\,\dd s\leq\frac{K^2d}{2\lambda} $ and $\|e^{-\frac{\scrm}{m}t_0}p\|\leq Kd|p|e^{-\lambda t_0}$, thanks to the assumption that $\max\left\{\left\|e^{-\scrm^*\delta}\right\|,\,\left\|e^{-\scrm\delta}\right\|\right\} \leq K e^{-\lambda\delta}$ for each $\delta>0$. Eq.~\eqref{IR REG bound, final} implies that $\left\|\ir^\mm_{n,{\rm REG}}(t,p)\right\| \lesssim \mathcal{O}(m)$ in the massless limit.

\smallskip
\noindent
\underline{Estimate for $\ir^\mm_{n,{\rm SING}}$.}  This term is nonzero only when $k=n / 2$ for $n$ even. In this case, 
\begin{align*}
\mathcal{R}_{n,{\rm SING}}^\mm(\tau_0,p) = \frac{n!}{\left(\frac{n}{2}\right)! \sqrt{2}^n} \left(\int_0^{\tau_0} e^{-\mom s}e^{-\frac{\scrm^*}{m}s} \,\dd s\right)^{\otimes {n \slash 2}}= \left(\sqrt{m}\right)^n \frac{n!}{\left(\frac{n}{2}\right)! \sqrt{2}^n} \left({\bf \Sigma}_{\tau_0\slash m}\right)^{\otimes (n\slash 2)}.
\end{align*}
Using the same estimates for $\left\|\mathcal{J}_{n-1}^\mm(t_0)\right\|$, $\left\|{\bf \Sigma}_{t_0}\right\|$ in the bound for $\ir^\mm_{n,{\rm REG}}$, we deduce that
\begin{footnotesize}
\begin{align}\label{IR Sing estimate, final}
&\left\|\ir^\mm_{n,{\rm SING}}(t,p)\right\| \leq \frac{\Lambda^nK^{n-1}d^n e^{-\lambda}}{(n-1)! \lambda^{n-1}} \int_0^{t/m} \left(\sqrt{m}\right)^n \frac{n!}{\left(\frac{n}{2}\right)! \sqrt{2}^n} \left\|{\bf \Sigma}_{t_0}^{\otimes (n\slash 2)}\right\| \,\dd t_0 \nonumber \\
&\leq \frac{\Lambda^nK^{n-1}d^n e^{-\lambda}}{(n-1)! \lambda^{n-1}} \cdot \left(\sqrt{m}\right)^n \frac{n!}{\left(\frac{n}{2}\right)! \sqrt{2}^n}  \int_0^{t/m}  \left(\frac{K^2d}{2\lambda}
\right)^{n/2} \,\dd t_0 
\leq \left(\frac{\Lambda}{2}\right)^n \frac{nK^{2n-1}d^{\frac{3n}{2}}e^{-\lambda}}{\lambda^{\frac{3n}{2}-1} \left(\frac{n}{2}\right)!} t \cdot m^{\frac{n}{2}-1}.
\end{align}
\end{footnotesize}
In particular, $\ir^\mm_{n,{\rm SING}}(t,p)$ vanishes in the massless limit whenever $n > 2$.

\smallskip

\noindent
\underline{Error bound.} Recall the decomposition $\two_{n,1}^\mm(t,p) = \fraka_n^\mm(t,p) + \ir_{n, {\rm REG}}^\mm(t,p) + \ir_{n, {\rm SING}}^\mm(t,p)$ in Eq.~\eqref{II 1, decomp}. From Eqs.~\eqref{IR REG bound, final} and \eqref{IR Sing estimate, final}, we deduce that 
\begin{align}\label{bound for II, 1}
&\left\|\two_{n,1}^\mm(t,p) - \fraka_n^\mm(t,p)\right\|
\leq \frac{\Lambda^nK^{2n} d^{2n} e^{-\lambda}}{(n-1)! \lambda^{n}} \sum_{1 \leq k < \lfloor n\slash 2 \rfloor} \frac{(2k)!|p|^{n-2k}}{k! 2^{2k} \lambda^k(n-2k)}\cdot m^k\nonumber\\
&\qquad\qquad\qquad\qquad\qquad\qquad + \left(\frac{\Lambda}{2}\right)^n \frac{nK^{2n-1}d^{\frac{3n}{2}}e^{-\lambda}}{\lambda^{\frac{3n}{2}-1} \left(\frac{n}{2}\right)!} t \cdot m^{\frac{n}{2}-1} {\bf 1}_{\{\text{$n$ is even}\}}\nonumber\\
&\qquad\qquad\leq C\left(K,\Lambda, \lambda^{-1},n,t\right) \left(1+|p|^{n-2}\right) \cdot m\quad \text{whenever } n \geq 3.
\end{align}

This concludes the estimates for $\two_{n, 1}^\mm$.

\subsubsection{Estimate for $\two_{n, 3}^\mm$, $\two_{n, 6}^\mm$, $\two_{n, 9}^\mm$} We check that these terms are all  $\lesssim \mathcal{O}(m)$ as $m \to 0^+$. 

\noindent
\underline{$\two_{n, 9}^\mm$.} This is easy --- recall that $
 \two_{n, 9}^\mm(t,p) = \frac{1}{2}\sum_{j=1}^d e_j \otimes e_j \otimes   \E^p\left[\int_0^t \err_{n-2}^\mm(t-s, P_s) \,\dd s\right]$. By induction, $\left\|\two_{n, 9}^\mm(t,p)\right\| \leq C\left(K,\Lambda, \lambda^{-1}, n,d,t\right) m \left\{\sup_{0 \leq \ell \leq n-4} \int_0^t \left\|\E^p \left[P_s^{\otimes \ell}\right] \right\|\,\dd s\right\}.$
Then it follows from Lemma~\ref{lem: pure expectation is O(m)} (see Appendix~\ref{sec: appendix, lemma} for a proof) that
\begin{align}\label{II 9}
    \left\|\two_{n, 9}^\mm(t,p)\right\| &\leq C\left(K,\Lambda, \lambda^{-1}, n,d,t\right)\left(1+|p|^{\max\{n-4,\,0\}}\right) m^2.
\end{align}

\begin{lemma}\label{lem: pure expectation is O(m)}
For each $\ell \in \{1,\ldots,n\}$ we have $
\int_0^t \left\|\E^p \left[P_s^{\otimes \ell}\right] \right\|\,\dd s \leq C\left(K,\Lambda, \lambda^{-1}, \ell\right)\left(1+|p|^\ell\right) m$. 
\end{lemma}


\noindent
\underline{$\two_{n, 3}^\mm$ and $\two_{n, 6}^\mm$.} 
Estimates for these terms
are similar. We only sketch the main steps.

For $\two_{n, 6}^\mm$, the crucial observation is that  $\na_p \err_{n-1}^\mm(t,p)$ is a polynomial of degree $\leq (n-4)$ in $p$, since $\err_{n-1}^\mm(t,p)$ is a polynomial of degree $\leq (n-3)$ in $p$. Moreover, in view of the inductive hypothesis~\eqref{ind hypo, xx}, the coefficients of $\na_p \err_{n-1}^\mm(t,p)$ are also of order $\mathcal{O}(m)$, modulo constants depending only on $K$, $\Lambda$, $\lambda^{-1}$, $n$, $d$, and $t$. The degree-1 (in terms of $p$) terms in $\err_{n-1}^\mm$ --- hence the constant terms in $\p_{p_j}\err_{n-1}^\mm$ --- contribute to $\mathcal{O}(m)$ in $\two_{n, 6}^\mm(t,p)$. On the other hand, the terms in $\err_{n-1}^\mm$ of $\deg_p \geq 2$, by virtue of Lemma~\ref{lem: pure expectation is O(m)}, will lead to $\mathcal{O}\left(m^2\right)$ terms in $\two_{n, 6}^\mm(t,p)$, with constants depending on the same parameters. Therefore, we obtain that
\begin{align}\label{II 6}
\left\|\two_{n, 6}^\mm(t,p)\right\| &\leq C\left(K,\Lambda, \lambda^{-1}, n,d,t\right)\left(1+|p|^{\max\{n-4,\,0\}}\right) m.
\end{align}

For $\two_{n, 3}^\mm$, we again use the inductive hypothesis~\eqref{ind hypo, xx} to find that $P_s \otimes \err_{n-1}^\mm(t-s, P_s)$ is a polynomial of degree $\leq (n-2)$ in $P_s$, and that the coefficients in this tensor are of order $\mathcal{O}(m)$ modulo constants depending only on $K$, $\Lambda$, $\lambda^{-1}$, $n$, $d$, and $t$. Then, by Lemma~\ref{lem: pure expectation is O(m)} we have
\begin{align*}
\int_0^t \left\|\E^p\left[P_s \otimes \err_{n-1}^\mm(t-s, P_s)\right]\right\| \,\dd s\leq C\left(K,\Lambda, \lambda^{-1}, n,d,t\right)\left(1+|p|^{n-2}\right) m^2.
\end{align*}
Note that there is another factor $m^{-1}$ in $\two_{n, 3}^\mm$. Hence,
\begin{align}\label{II 3}
\left\|\two_{n, 3}^\mm(t,p)\right\|\leq C\left(K,\Lambda, \lambda^{-1}, n,d,t\right)\left(1+|p|^{n-2}\right) m.
\end{align}

\subsubsection{Estimate for $\two_{n, 4}^\mm$ and $\two_{n, 7}^\mm$} 
It is crucial that we have assumed $n\geq 3$ throughout this proof. 
Recall once again the notation $
\fraka_\ell^\mm(t,p) := \idotsint\limits_{0\leq t_1<\dots<t_\ell\leq t} \,\,\bigotimes_{j=1}^\ell  \left( -\mom e^{-\mom t_j} p\right)\, \dd t_1\ldots\,\dd t_\ell$.

For $\two^\mm_{n,7}$ we have the following lemma (see Appendix~\ref{sec: appendix, lemma} for a proof), which implies
\begin{align}\label{II n,7 estimate}
\left\|\two^\mm_{n,7}(t,p)\right\| \leq C\left(K,\Lambda, \lambda^{-1},n,d\right) \left(1+|p|^{n-2}\right) m.
\end{align}

\begin{lemma}\label{lem for II 7}
For any $\ell \geq 1$, it holds that
\begin{align*}
\E^p\left[\int_0^t \left\|\fraka_\ell^\mm(t-s,P_s)\right\| \,\dd s \right]  \leq C\left(K,\Lambda,\lambda^{-1},\ell,d\right)\left(1+|p|^\ell\right)\cdot m.
\end{align*}
\end{lemma}

The bound for $\two_{n,4}^\mm$ is  analogous. Observe that $\left(\na_p \fraka_{n-1}^\mm\right)(t,p)$ equals the iterated integral over $\blacktriangle_{n-1,\,t}$ of contractions of $\left( -\mom e^{-\mom t_j} \right)^{\otimes (n-1)} \otimes p^{\otimes (n-2)}$. Then, by 
Lemma~\ref{lem for II 7} (more precisely, the arguments thereof), we have that
\begin{align}\label{II n,4 estimate}
\left\|\two^\mm_{n,4}(t,p)\right\| \leq C\left(K,\Lambda, \lambda^{-1},n,d\right) \left(1+|p|^{n-2}\right) m.
\end{align}

\subsubsection{Estimate for $\two_{n, 5}^\mm$ and $\two_{n, 8}^\mm$}\label{subsubsec: II 2,5,8}

We first deal with $\two_{n, 8}^\mm(t,p)$, and the estimate for $\two_{n, 5}^\mm$ is similar. Recall that, for each $q \in \mathbb{N}$,
\begin{align}\label{Bq expansion}
&\frakb_{q}^\mm(t-s,P_s) = \fraka_{q-2}^\mm(t-s, P_s) \otimes\left[ \Xi^\mm(t-s)\right]^{\otimes 1}(t-s)  + \fraka_{q-4}^\mm(t-s, P_s) \otimes \left[\Xi^\mm(t-s)\right]^{\otimes 2}(t-s)^2 \nonumber\\
&\quad  + \cdots + \fraka_{q-2 \cdot \lfloor\frac{ q}{2}\rfloor}^\mm(t-s, P_s) \otimes \left[\Xi^\mm(t-s)\right]^{\otimes \lfloor\frac{q}{2}\rfloor}(t-s)^{\otimes \lfloor\frac{q}{2}\rfloor},
\end{align}
where $\Xi^\mm(t-s) = \scrm \cdot \left(\dashint_0^{\frac{t-s}{m}}\ssig_\sigma\,\dd\sigma\right)-\frac{\id}{2}$.  In addition, by Eq.~\eqref{Xi estimate} and Lemma~\ref{lem: average difference}, when $t>s>0$ we have that $\left\|\Xi^\mm(t)-\Xi^\mm(t-s)\right\| \leq \frac{2\Lambda K^2d^2}{\lambda}\cdot\frac{s}{t}$. 
  
Our crucial observation for estimating $\two_{n, 8}^\mm$ (similarly for $\two_{n, 5}^\mm$, $\two_{n, 2}^\mm$) is that all the terms in the expansion of $\frakb_{q}^\mm(t-s,P_s)$ will contribute to $\mathcal{O}(m)$, except for the constant term (degree-0 term in $P_s$). But the lowest degree term in $P_s$ is $$\fraka_{q-2 \cdot \lfloor\frac{ q}{2}\rfloor}^\mm(t-s, P_s) \otimes \left[\Xi^\mm(t-s)\right]^{\otimes \lfloor\frac{q}{2}\rfloor}(t-s)^{\lfloor\frac{q}{2}\rfloor},$$ so the degree-0 terms are present only when $q-2 \cdot \lfloor\frac{q}{2}\rfloor=0$; \emph{i.e.}, when $q$ is even.

\begin{lemma}\label{lem: handling B in II 2,5,8}
For each $\ell = 1,2,3,\ldots$ and $r = 0,1,2,\ldots$, it holds that\footnote{It will be clear from the proof that if both $\ell$ and $r$ $\leq n \in \mathbb{N}$, then the constant $C=C\left(K,\Lambda, \lambda^{-1},n,d,t\right)$.}

$\left\|\E^p\left[ \int_0^t \fraka_\ell(t-s, P_s) \otimes \left(\Xi^\mm(t-s)\right)^{\otimes r} (t-s)^r \,\dd s\right]\right\|  \leq C\left(K,\Lambda, \lambda^{-1},\ell,r,d,t\right) \left(1+|p|^\ell\right)\cdot m.$
\end{lemma}
Lemma~\ref{lem: handling B in II 2,5,8} (whose proof is in Appendix~\ref{sec: appendix, lemma}) directly implies that
\begin{align}\label{II 8, n odd}
\left\|\two_{n,8}^\mm(t,p)\right\| \leq C\left(K,\Lambda, \lambda^{-1},n,d,t \right) \left(1+|p|^{n-2}\right)\cdot m\qquad \text{ if $n$ is odd};
\end{align}
\begin{align}\label{II 8, n even}
&\Bigg\|\two_{n,8}^\mm(t,p) - \frac{1}{2} \underbrace{ \sum_{j=1}^d e_j \otimes e_j }_{\equiv \, \id} \otimes \left\{\int_0^t \left[\Xi^\mm(t-s)\right]^{\otimes  \left(\frac{n}{2}-1\right)}(t-s)^{\left(\frac{n}{2}-1\right)} \,\dd s\right\} 
\Bigg\|\nonumber \\
&\qquad\leq C\left(K,\Lambda, \lambda^{-1},n,d,t \right) \left(1+|p|^{n-2}\right)\cdot m \qquad \text{ if $n$ is even}.
\end{align}
Thus, we have obtained desired estimates for $\two_{n,8}$. 

Now we consider $\two_{n,5}^\mm$. For odd $n$, taking $q=n-1$ in Eq.~\eqref{Bq expansion} shows that $\left(\na_{p}\frakb_{n-1}^\mm\right)(t-s,P_s)$ consists of \emph{odd}-degree linear combinations of $P_s$. Thus, similar arguments as for $\two^\mm_{n,8}$ (employing the structure of $\fraka_\bullet^\mm$ and adaptations of the proof for Lemmas~\ref{lem for II 7} $\&$ \ref{lem: handling B in II 2,5,8}) lead to 
\begin{align}\label{II 5, n odd}
\left\|\two_{n,5}^\mm(t,p)\right\| \leq C\left(K,\Lambda, \lambda^{-1},n,d,t \right) \left(1+|p|^{n-4}\right)\cdot m\qquad \text{ if $n$ is odd}, 
\end{align}
On the other hand, if $n$ is even, the same arguments above show that there is one $\mathcal{O}(1)$ term arising from $\left(\p_{p_j}\fraka_1^\mm\right)(t-s,p)$. More precisely,

\begin{align*}
&\Bigg\|\two_{n,5}^\mm(t,p) - \Bigg\{\sum_{j=1}^d e_j \otimes   \E^p\Bigg[\int_0^t  \left(\p_{p_j}\fraka_{1}^\mm\right)(t-s, P_s) \otimes\left[\Xi^\mm(t-s)\right]^{\otimes \left(\frac{n}{2}-1\right)} (t-s)^{\frac{n}{2}-1}\,\dd s\Bigg]\Bigg\}\Bigg\| \\
&\qquad\qquad \leq C\left(K,\Lambda, \lambda^{-1},n,d,t \right) \left(1+|p|^{n-2}\right)\cdot m\qquad \text{ if $n$ is even}.
\end{align*}
Direct computation shows that $\left[\left(\p_{p_j}\fraka_{1}^\mm\right)\right]^\beta(t-s,p)  = \left[e^{-\mom (t-s)} - \id\right]_j^{\beta}$; in particular, the derivative is independent of the $P_s$ variable.\footnote{The left-hand side is the $\beta^{\text{th}}$-component of$\left[\left(\p_{p_j}\fraka_{1}^\mm\right)\right](t-s,p)$, which is a $d$-vector. Moreover, given a matrix $\mathfrak{m} =\left\{\mathfrak{m}^{i}_j\right\} \in \mathfrak{gl}(d;\C)$, for each fixed $k \in \{1,\ldots,d\}$ we denote by $\mathfrak{m}^{k}_{\bullet}$ the vector $\left(\mathfrak{m}^{k}_{1}, \ldots, \mathfrak{m}^{k}_{d}\right)^\top \in \C^d$.}  Thus
\begin{align}\label{II 5, n even}
&\Bigg\|\two_{n,5}^\mm(p,t) - \Bigg\{\sum_{j=1}^d e_j \otimes \Bigg[ \int_0^t \left[e^{-\mom (t-s)} - \id\right]^{j}_{\bullet}\otimes \left[\Xi^\mm(t-s)\right]^{\otimes \left(\frac{n}{2}-1\right)} (t-s)^{\frac{n}{2}-1}\,\dd s\Bigg]\Bigg\}\Bigg\|\nonumber \\
&\qquad\qquad \leq C\left(K,\Lambda, \lambda^{-1},n,d,t \right) \left(1+|p|^{n-2}\right)\cdot m\qquad \text{ if $n$ is even}.
\end{align}
The left-hand side can be further simplified by Lemma~\ref{lem: comparing shift in time, Xi(t-m.)} below.

\subsubsection{Estimate for $\two_{n, 2}^\mm$}\label{subsubsec: II 2} Our estimates are parallel in the large to those for $\two_{n,1}^\mm$ in \S\ref{subsubsec: II n,1}.  Indeed, by definition of $\frakb_{n-1}^\mm$, it holds that 
\begin{align*}
\two_{n,2}^\mm(t,p) = \sum_{j=1}^{\lfloor \frac{n-1}{2}
\rfloor}\E^p\Bigg[\int_0^t \left(-\mom P_s\right)\otimes \fraka_{n-1-2j}^{\mm}(t-s, P_s)\left[\Xi^\mm(t-s)\right]^{\otimes j} (t-s)^j\,\dd s \Bigg].
\end{align*}
In the above, for each $\ell = n-1-2j$ we have
\begin{align*}
\left(-\scrm P_s\right)\otimes \fraka_\ell^\mm (t-s, P_s) = \left\{\,\,\idotsint\limits_{0\leq t_1<\dots<t_\ell\leq \frac{t-s}{m}} \,\,(-\scrm)\otimes \bigotimes_{j=1}^\ell  \left( -\scrm e^{-\scrm t_j} \right)\, \dd t_1\ldots\,\dd 
 t_\ell\right\}\cdot P_s^{\otimes 
 (\ell+1)}.
\end{align*}
Thanks to Lemma~\ref{lem:GaussianMoment}, one may further express (with $C_k \equiv \frac{(2k)!}{k! 2^k}$ from now on)
\begin{eqnarray*}
&&\E^p\left[P_s^{\otimes (\ell+1)}\right] = \left( e^{-\mom s} p \right)^{\otimes (\ell+1)} +\mathcal{R}_{\ell+1}^\mm(s,p),\\
&&\mathcal{R}_{\ell+1}^\mm(s,p) = \sum_{k=1}^{\lfloor \frac{\ell+1}{2} \rfloor} C_k \,\sym\left\{\left(e^{-\mom s}p\right)^{\otimes (\ell+1-2k)} \otimes \left(\int_0^s e^{-\mom \varsigma} e^{-\frac{\scrm^*}{m}\varsigma}\,\dd\varsigma \right)^{\otimes k}\right\}.
\end{eqnarray*}

To resume, by considering $\sigma:=\frac{s}{m}$ and taking $\E^p$, we obtain that
\begin{align*}
\two_{n,2}^\mm(t,p) &= \sum_{j=1}^{\lfloor \frac{n-1}{2}
\rfloor}\int_0^{\frac{t}{m}}  \left\{\,\,\idotsint\limits_{0\leq t_1<\dots<t_\ell\leq \frac{t}{m}-\sigma} \,\,(-\scrm)\otimes \bigotimes_{h=1}^{\ell=n-1-2j}  \left(-\scrm e^{-\scrm t_h} \right)\, \dd t_1\ldots\,\dd 
 t_\ell\right\} 
\nonumber\\
&\qquad\cdot\left\{ \left(e^{-\scrm \sigma}p\right)^{\otimes (\ell+1)} + \mathcal{R}^\mm_{\ell+1}(m\sigma,p) \right\}\otimes \left[\Xi^\mm(t-m\sigma)\right]^{\otimes j} (t-m\sigma)^j\,\dd \sigma,
\end{align*}
where
\begin{align*}
\mathcal{R}_{\ell+1}^\mm(m\sigma,p) = \sum_{k=1}^{\lfloor \frac{\ell+1}{2} \rfloor} C_k\,\sym\left\{\left(e^{-\scrm\sigma}p\right)^{\otimes (\ell+1-2k)} \otimes \left(\int_0^\sigma e^{-\scrm \varsigma} e^{-{\scrm^*}\varsigma}\,\dd\varsigma \right)^{\otimes k}\right\}\cdot m^k.
\end{align*}

Let us write $    \two_{n,2}^\mm(t,p) = \two_{n,2,1}^\mm(t,p) + \two_{n,2,2}^\mm(t,p)$ where \begin{align*}
\two_{n,2,1}^\mm(t,p) &:=  \sum_{j=1}^{\lfloor \frac{n-1}{2}
\rfloor}\E^p\Bigg[\int_0^{\frac{t}{m}}  \left\{\,\,\idotsint\limits_{0\leq t_1<\dots<t_\ell\leq \frac{t}{m}-\sigma} \,\,(-\scrm)\otimes \bigotimes_{h=1}^{\ell:=n-1-2j}  \left( -\scrm e^{-\scrm t_h} \right)\, \dd t_1\ldots\,\dd 
 t_\ell\right\} 
\nonumber\\
&\qquad\cdot \left(e^{-\scrm \sigma}p\right)^{\otimes (\ell+1)} \otimes \left[\Xi^\mm(t-m\sigma)\right]^{\otimes j} (t-m\sigma)^j\,\dd \sigma \Bigg];
\end{align*}
\begin{align*}
\two_{n,2,2}^\mm(t,p) &:=  \sum_{j=1}^{\lfloor \frac{n-1}{2}
\rfloor}\int_0^{\frac{t}{m}}  \left\{\,\,\idotsint\limits_{0\leq t_1<\dots<t_\ell\leq \frac{t}{m}-\sigma} \,\,(-\scrm)\otimes \bigotimes_{h=1}^{\ell=n-1-2j}  \left( -\scrm e^{-\scrm t_h} \right)\, \dd t_1\ldots\,\dd 
 t_\ell\right\} 
\nonumber\\
&\qquad\cdot  \mathcal{R}^\mm_{\ell+1}(p, m\sigma) \otimes \left[\Xi^\mm(t-m\sigma)\right]^{\otimes j} (t-m\sigma)^j\,\dd \sigma.  
\end{align*}
\smallskip
\noindent
\underline{$\two_{n,2,1}^\mm(t,p)$.} We introduce new variables $\tau_0 := \sigma$ and $\tau_j := t_j + \tau_0$ for each $j \in \{1,\ldots,\ell\}$ to deduce  
\begin{align*}
\two_{n,2,1}^\mm(t,p) &= \sum_{j=1}^{\lfloor \frac{n-1}{2}\rfloor} \Bigg\{\,\,\idotsint\limits_{0\leq \tau_0 <\ldots<\tau_\ell\leq \frac{t}{m}} \bigotimes_{h=0}^{\ell:=n-1-2j}  \left( -\scrm e^{-\scrm \tau_h}p \right) \otimes\left[\Xi^\mm(t-m\tau_0 )\right]^{\otimes j} (t-m\tau_0 )^j \,\dd \tau_0\ldots\dd\tau_\ell\Bigg\}.
\end{align*}

Assume  Lemma~\ref{lem: comparing shift in time, Xi(t-m.)} below for the moment. One may infer that
\begin{align}\label{II n,2,1 estimate}
&\two_{n,2,1}^\mm(t,p) = {\rm Error_{n,2,1}}(t,p) \\
&\quad +  \sum_{j=1}^{\lfloor \frac{n-1}{2}\rfloor} \left\{\,\,\idotsint\limits_{0\leq \tau_0 <\ldots<\tau_\ell\leq \frac{t}{m}} \bigotimes_{h=0}^{\ell}  \left( -\scrm e^{-\scrm \tau_h}p \right) \dd \tau_0\ldots\dd\tau_\ell\right\} \otimes \left[\Xi^\mm(t )\right]^{\otimes j} t^j \quad \text{ where $\ell:=n-1-2j$,} \nonumber
\end{align}
with the error term controlled by 
\begin{align}\label{II n,2,1 error}
\Big\|{\rm Error_{n,2,1}}(p,t)\Big\| \leq C\left(K,\Lambda,\lambda^{-1},n,d,t\right)\left(1+|p|^{n-3}\right)\cdot m.
\end{align} 

\begin{lemma}\label{lem: comparing shift in time, Xi(t-m.)}
Consider as before $\Xi^\mm(s):=\scrm\cdot\left(\dashint_0^{s\slash m} {\ssig}_\varsigma\,\dd\varsigma\right)-\frac{\id}{2}.$ For each $j \in \mathbb{N}$, $m \in ]0,1]$, and $\tau_0>0$ such that $t-m\tau_0 > 0$, it holds that 
\begin{align*}
\left\|\left[\Xi^\mm(t-m\tau_0 )\right]^{\otimes j} (t-m\tau_0 )^j - \left[\Xi^\mm(t )\right]^{\otimes j} t^j\right\| \leq C\left(K,\Lambda,\lambda^{-1},j,d,t\right)\cdot m.
\end{align*}
\end{lemma}

Since $\lfloor \frac{n-1}{2}\rfloor = \frac{n}{2}-1$ for $n$ even, the second term on the RHS of  Eq.~\eqref{II n,2,1 estimate} equals
\begin{align*}
    &\sum_{j=1}^{\lfloor \frac{n-1}{2}\rfloor}\left(\,\,\idotsint\limits_{0\leq \tau_0 <\ldots<\tau_\ell\leq \frac{t}{m}} \bigotimes_{h=0}^{\ell:=n-1-2j}  \left( -\scrm e^{-\scrm \tau_h}p \right) \dd \tau_0\ldots\dd\tau_\ell\right) \otimes \left[\Xi^\mm(t )\right]^{\otimes j} t^j \\
    &=
    \begin{cases} \fraka_{n-2}(t,p) \otimes \left[\Xi^\mm(t)\right]^{\otimes 1} t^1 + \cdots + \fraka_{1}(t,p) \otimes \left[\Xi^\mm(t)\right]^{\otimes \frac{n-1}{2}} t^{\frac{n-1}{2}} \text{ for odd $n$},\\
    \fraka_{n-2}(t,p) \otimes \left[\Xi^\mm(t)\right]^{\otimes 1} t^1 + \cdots + \fraka_{2}(t,p) \otimes \left[\Xi^\mm(t)\right]^{\otimes \frac{n}{2}-1} t^{\frac{n}{2}-1} \text{ for even $n$}.
    \end{cases}
\end{align*} 
Thus, in view of the definition of $\frakb_q^\mm$, we have 
 \begin{align*}
    &\sum_{j=1}^{\lfloor \frac{n-1}{2}\rfloor}\left(\,\,\idotsint\limits_{0\leq \tau_0 <\ldots<\tau_\ell\leq \frac{t}{m}} \bigotimes_{h=0}^{\ell:=n-1-2j}  \left( -\scrm e^{-\scrm \tau_h}p \right) \dd \tau_0\ldots\dd\tau_\ell\right) \otimes \left[\Xi^\mm(t )\right]^{\otimes j} t^j \\
    &\qquad =
    \begin{cases} \frakb_{n}(t,p) \qquad \text{ for odd $n$},\\
    \frakb_{n}(t,p)  - \left[\Xi^\mm(t)\right]^{\otimes \frac{n}{2}} t^{\frac{n}{2}} \qquad \text{ for even $n$}.
    \end{cases}
\end{align*} 

We may thus reformulate Eq.~\eqref{II n,2,1 estimate} as
\begin{align}\label{II n,2,1 estimate, UPDATED}
\two_{n,2,1}^\mm(t,p) = {\rm Error_{n,2,1}}(t,p) +   \begin{cases}
\frakb_{n}(t,p) \quad \text{ for odd $n$},\\
    \frakb_{n}(t,p)  - \left[\Xi^\mm(t)\right]^{\otimes \frac{n}{2}} t^{\frac{n}{2}} \quad \text{ for even $n$}.
 \end{cases}
\end{align}

\noindent
\underline{$\two_{n,2,2}^\mm(t,p)$.} Similar to the estimate for $\two_{n,1}$ (Eqs.~\eqref{II 1, decomp}, \eqref{IR def}), we split the term $\mathcal{R}_{\ell+1}^\mm$ into 
\begin{align*}
&\mathcal{R}_{\ell+1}^\mm(m\sigma,p) = \mathcal{R}_{\ell+1,{\rm REG}}^\mm(m\sigma,p) + \mathcal{R}_{\ell+1,{\rm SING}}^\mm(m\sigma,p)\\
&\,\,:= \sum_{k=1}^{\lfloor \frac{\ell+1}{2} \rfloor} C_k\,\sym\left\{\left(e^{-\scrm\sigma}p\right)^{\otimes (\ell+1-2k)} \otimes \left(\int_0^\sigma e^{-\scrm \varsigma} e^{-{\scrm^*}\varsigma}\,\dd\varsigma \right)^{\otimes k}\right\}\cdot m^k \1_{\{\ell+1-2k > 0\}} \\
&\quad + \sum_{k=1}^{\lfloor \frac{\ell+1}{2} \rfloor} C_k\,\sym\left\{\left(e^{-\scrm\sigma}p\right)^{\otimes (\ell+1-2k)} \otimes \left(\int_0^\sigma e^{-\scrm \varsigma} e^{-{\scrm^*}\varsigma}\,\dd\varsigma \right)^{\otimes k}\right\}\cdot m^k \1_{\{\ell+1-2k = 0\}}.
\end{align*}
Recall $\ell = n-1-2j$ and $C_k = \frac{(2k)!}{k! 2^k}$. Set accordingly
$\two_{n,2,2}^\mm(t,p) = \two_{n,2,2, {\rm REG}}^\mm(t,p) + \two_{n,2,2, {\rm SING}}^\mm(t,p)$, where for $\bullet \in \{{\rm REG},\,{\rm SING}\}$ we put 
\begin{align*}
\two_{n,2,2, {\bullet}}^\mm(t,p)   &:=  \sum_{j=1}^{\lfloor \frac{n-1}{2}
\rfloor}\int_0^{\frac{t}{m}}  \left\{\,\,\idotsint\limits_{0\leq t_1<\dots<t_\ell\leq \frac{t}{m}-\sigma} \,\,(-\scrm)\otimes \bigotimes_{h=1}^{\ell}  \left( -\scrm e^{-\scrm t_h} \right)\, \dd t_1\ldots\,\dd 
 t_\ell\right\} 
\nonumber\\
&\qquad\cdot  \mathcal{R}^\mm_{\ell+1, \bullet}(p, m\sigma) \otimes \left[\Xi^\mm(t-m\sigma)\right]^{\otimes j} (t-m\sigma)^j\,\dd \sigma.
\end{align*}

We first bound the regular term $\two_{n,2,2, {\rm REG}}^\mm$. Note that
\begin{align*}
\left\|\mathcal{R}^\mm_{\ell+1,{\rm REG}}(m\sigma,p)
    \right\| \leq C\left(K,\lambda^{-1}, \ell,d\right) \left(1+|p|^{\ell-1}\right) e^{- \frac{1}{2} \lambda \sigma}\cdot m,\quad \ell=n-1-2j,
\end{align*}
for which we have used the facts that $\ell+1-2k \geq \frac{1}{2}$ for all the summands in ${\mathcal{R}^\mm_{\ell+1,{\rm REG}}(m\sigma,p)}$, and $\left\|\ssig_\sigma\right\| \leq \frac{K^2d}{2\lambda}$ for any $\sigma>0$ (Eq.~\eqref{Sigma estimate, simple}). The exponentially decaying factor $e^{- \frac{1}{2} \lambda \sigma}$ gives us
\begin{align*}
&\Bigg\| \int_0^{\frac{t}{m}}  \left\{\,\,\idotsint\limits_{0\leq t_1<\dots<t_\ell\leq \frac{t}{m}-\sigma} \,\,(-\scrm)\otimes \bigotimes_{h=1}^{\ell}  \left( -\scrm e^{-\scrm t_h} \right)\, \dd t_1\ldots\,\dd
 t_\ell\right\} \cdot\\
 &\qquad\qquad  \cdot {\mathcal{R}^\mm_{\ell+1,{\rm REG}}(m\sigma,p)}\,\dd \sigma\Bigg\| \leq C\left(K,\Lambda,\lambda^{-1},n,d, j\right)\left(1+|p|^{\ell-1}\right) \cdot m,
\end{align*}
by rewriting the left-hand side as an iterated integral over $\blacktriangle_{\ell+1, \frac{t}{m}}:=\left\{0\leq\tau_0<\tau_1< \ldots< \tau_\ell \leq \frac{t}{m}\right\}$ with $\tau_0 := \sigma$ and $\tau_j := \sigma + t_j$ for  $j \in \{1,\ldots,\ell\}$, and then employing the  trick of symmetrisation in Lemma~\ref{lem: sym}. Using Lemma~\ref{lem: comparing shift in time, Xi(t-m.)} and estimating $\left\|\Xi^\mm(t)\right\|$ via Eq.~\eqref{Xi m, uniform bd}, we thus bound
\begin{align}\label{II n,2,2, REG, final bound}
\left\|\two^\mm_{n,2,2,{\rm REG}}(t,p)\right\| &\leq C\left(K,\lambda^{-1},\Lambda, n,t,d\right) \left(1+|p|^{n-2}\right)\cdot m,
\end{align}
since $\ell := n-1-2j$ with $j \geq 1$.

Finally, we turn to $\two_{n,2,2, {\rm SING}}^\mm$, which is independent of $p$. Observe that $\mathcal{R}^\mm_{\ell+1,{\rm SING}}(m\sigma,p) = C_{\frac{\ell+1}{2}} \left(\ssig_{\sigma}\right)^{\otimes \left(\frac{\ell+1}{2}\right)} \cdot m^{\frac{\ell+1}{2}} \,\1_{\{\text{$\ell$ is odd}\}}$. As $\ell := n-1-2j$ where $j \in \left[ 1, \lfloor \frac{n-1}{2} \rfloor \right] \cap \mathbb{Z}$, we have
\begin{align*}
\mathcal{R}^\mm_{\ell+1,{\rm SING}}(m\sigma,p) = C_{\frac{n}{2} -j} \left(\ssig_{\sigma}\right)^{\otimes \left(\frac{n}{2} -j\right)} \cdot m^{\frac{n}{2} -j} \,\1_{\{\text{$n$ is even}\}}
\end{align*}
for $j \in \left\{1,\ldots, \frac{n}{2}-1\right\}$. We consider here $n \geq 4$ only. It follows that
\begin{align*}
\two_{n,2,2, {\rm SING}}^\mm(t,p)&=  \sum_{j=1}^{\frac{n}{2}-1}\int_0^{\frac{t}{m}}  \left\{\,\,\idotsint\limits_{0\leq t_1<\dots<t_{n-1-2j}\leq \frac{t}{m}-\sigma} \,\,(-\scrm)\otimes \bigotimes_{h=1}^{n-1-2j}  \left( -\scrm e^{-\scrm t_h} \right)\, \dd t_1\ldots\,\dd 
 t_{n-1-2j}\right\} 
\nonumber\\
&\qquad\cdot   C_{\frac{n}{2} -j} {\bf \Sigma}_{\sigma}^{\otimes \left(\frac{n}{2} -j\right)} \cdot m^{\frac{n}{2} -j}  \otimes \left[\Xi^\mm(t-m\sigma)\right]^{\otimes j} (t-m\sigma)^j\,\dd \sigma.
\end{align*}
Separating the term with $j=\frac{n}{2}-1$ from the rest, let us further decompose
\begin{align*}
&\two_{n,2,2, {\rm SING}}^\mm(t,p) = \two_{n,2,2, {\rm SING}, {\rm BAD}}^\mm(t,p)+\two_{n,2,2, {\rm SING},\,{\rm GOOD}}^\mm(t,p)\\  
&\quad:= m \cdot \int_0^{\frac{t}{m}} \left(\int_0^{\frac{t}{m}-\sigma} (-\scrm)\otimes(-\scrm e^{-\scrm t_1})\,\dd t_1\right)\cdot{\bf \Sigma}_\sigma \otimes \left[\Xi^\mm(t-m\sigma) \right]^{\otimes \left( \frac{n}{2}-1 \right)}(t-m\sigma)^{\frac{n}{2}-1}\,\dd \sigma\\
&\qquad +  \sum_{j=1}^{\frac{n}{2}-2}\int_0^{\frac{t}{m}}  \left\{\,\,\idotsint\limits_{0\leq t_1<\dots<t_{n-1-2j}\leq \frac{t}{m}-\sigma} \,\,(-\scrm)\otimes \bigotimes_{h=1}^{n-1-2j}  \left( -\scrm e^{-\scrm t_h} \right)\, \dd t_1\ldots\,\dd 
 t_{n-1-2j}\right\} \cdot 
\nonumber\\
&\qquad\quad\cdot   C_{\frac{n}{2} -j} {\bf \Sigma}_{\sigma}^{\otimes \left(\frac{n}{2} -j\right)} \cdot m^{\frac{n}{2} -j}  \otimes \left[\Xi^\mm(t-m\sigma)\right]^{\otimes j} (t-m\sigma)^j\,\dd \sigma \qquad \text{for $n \geq 4$ even}.
\end{align*}

The bound for the good term is similar to that for $\two^\mm_{n,2,2,{\rm REG}}$ above. Indeed,  the symmetrisation trick in Lemma~\ref{lem: sym} leads to
\begin{align*}
\left\|\,\,\,\idotsint\limits_{0\leq t_1<\dots<t_{n-1-2j}\leq \frac{t}{m}-\sigma} (-\scrm)\otimes \bigotimes_{h=1}^{n-1-2j}  \left( -\scrm e^{-\scrm t_h} \right)\, \dd t_1\ldots\,\dd 
 t_{n-1-2j}\right\| \leq C\left(K, \Lambda, \lambda^{-1},n,d\right).
\end{align*}
Eq.~\eqref{Sigma estimate, simple} leads to $\left\| \left(\ssig_{\sigma}\right)^{\otimes \left(\frac{n}{2} -j\right)} \right\| \leq C\left(K,\lambda^{-1},n,d\right)$. Also, Lemma~\ref{lem: comparing shift in time, Xi(t-m.)} and Eq.~\eqref{Xi m, uniform bd} lead to
$\left\|\left[\Xi^\mm(t-m\sigma)\right]^{\otimes j} (t-m\sigma)^j\right\| \leq C\left(K,\Lambda,\lambda^{-1},n,d,t\right) (1+m)$. The factor $m^{\frac{n}{2}-j}$ contributes to $m^2$ and higher orders ($1 \leq j \leq \frac{n}{2}-2$), and the integral $\int_0^{{t}\slash{m}}\,\dd\sigma$ yields a factor of $m^{-1}$. Hence,
\begin{align}\label{II n,2,2, SING, GOOD; final bound}
\left\|\two_{n,2,2,\,{\rm SING},\,{\rm GOOD}}^\mm(t,p)\right\| \leq C\left(K,\Lambda,\lambda^{-1},n,d,t\right) \cdot m.
\end{align}

It remains to bound $\two_{n,2,2,\,{\rm SING},\,{\rm BAD}}^\mm$. As $    \int_0^{\frac{t}{m}-\sigma} -\scrm e^{-\scrm t_1}\,\dd t_1 = -\id + e^{-\scrm\left(\frac{t}{m}-\sigma\right)}$, one has
\begin{align*}
\two_{n,2,2,\,{\rm SING},\,{\rm BAD}}^\mm(t,p) 
&= m \cdot \int_0^{\frac{t}{m}} \left(\scrm \ssig_\sigma\right)\otimes \left[\Xi^\mm(t-m\sigma) \right]^{\otimes \left( \frac{n}{2}-1 \right)}(t-m\sigma)^{\frac{n}{2}-1}\,\dd \sigma\\
&\quad + m \cdot \int_0^{\frac{t}{m}} (-\scrm)\otimes e^{-\scrm\left(\frac{t}{m}-\sigma\right)} \cdot\ssig_\sigma\otimes \left[\Xi^\mm(t-m\sigma) \right]^{\otimes \left( \frac{n}{2}-1 \right)}(t-m\sigma)^{\frac{n}{2}-1}\,\dd \sigma,
\end{align*}
thanks to Fubini's theorem. For the second line of the right-most term,  note  by Lemma~\ref{lem: comparing shift in time, Xi(t-m.)},  the change of variables $\sigma' := \frac{t}{m}-\sigma$, and Eqs.~\eqref{Xi m, uniform bd} $\&$ \eqref{Sigma estimate, simple}   that 
\begin{align*}
&m \left\|\int_0^{\frac{t}{m}} (-\scrm)\otimes e^{-\scrm\left(\frac{t}{m}-\sigma\right)} \cdot\ssig_\sigma\otimes \left[\Xi^\mm(t-m\sigma) \right]^{\otimes \left( \frac{n}{2}-1 \right)}(t-m\sigma)^{\frac{n}{2}-1}\,\dd \sigma\right\|\\
&\qquad\leq C\left(K,\Lambda, \lambda^{-1}, n,d, t\right) \cdot m\,\left\|\int_0^{\frac{t}{m}} (-\scrm)\otimes e^{-\scrm\left(\frac{t}{m}-\sigma\right)} \cdot\ssig_\sigma \,\dd \sigma\right\|  \leq C\left(K,\Lambda, \lambda^{-1}, n,d, t\right) \cdot m.
\end{align*}

For the first line, we utilise once again  Lemma~\ref{lem: comparing shift in time, Xi(t-m.)} and Eq.~\eqref{Xi m, uniform bd} to express 
\begin{align*}
&m \cdot \int_0^{\frac{t}{m}} \left(\scrm \ssig_\sigma\right)\otimes \left[\Xi^\mm(t-m\sigma) \right]^{\otimes \left( \frac{n}{2}-1 \right)}(t-m\sigma)^{\frac{n}{2}-1}\,\dd \sigma\\
&\quad = \underbrace{m \cdot \left\{\int_0^{\frac{t}{m}} \left(\scrm \ssig_\sigma\right)\,\dd\sigma\right\} \otimes \left[\Xi^\mm(t)\right]^{\otimes \left(\frac{n}{2}-1\right)} t^{\frac{n}{2}-1}}_{=:\, \text{[LEAD]}} + C\left(\Lambda,\lambda^{-1},n,t\right)\cdot m^2.
\end{align*}
The nontrivial leading term can be further written as
\begin{align}\label{lead}
{\rm [LEAD]} = \scrm \cdot \left(\dashint_{0}^{\frac{t}{m}} \ssig_\sigma \,\dd \sigma \right) \otimes \left[\Xi^\mm(t)\right]^{\otimes \left(\frac{n}{2}-1\right)} t^{\frac{n}{2}}.    
\end{align}

In conclusion, we have obtained that
\begin{align}
    \label{II n,2,2,SING, BAD, final estimate}
\left\|\two_{n,2,2,\,{\rm SING},\,{\rm BAD}}^\mm(t,p)-{\rm [LEAD]}  \right\| \leq    C\left(K,\Lambda, \lambda^{-1}, n,d, t\right) \cdot m.
\end{align}

\subsubsection{Conclusion of the proof}  Recall from Eq.~\eqref{nine terms} that for each level $n \in \mathbb{Z}_{\geq 3}$, we have split  the expected signature $\Phi^\mm(p,t)$ into nine terms: $\two_{n,1}^\mm, \ldots, \two_{n,9}^\mm$. In \S\S\ref{subsubsec: II n,1}--\ref{subsubsec: II 2} we have estimated $\two_{n,1}^\mm$ in Eq.~\eqref{bound for II, 1};     $\two_{n,2}^\mm$ in Eqs.\eqref{II n,2,1 error} -\eqref{II n,2,2,SING, BAD, final estimate}; $\two_{n,3}^\mm$ in Eq.~\eqref{II 3}; $\two_{n,4}^\mm$ in Eq.~\eqref{II n,4 estimate}; $\two_{n,5}^\mm$ in Eqs.~\eqref{II 5, n odd} and \eqref{II 5, n even}; $\two_{n,6}^\mm$ in Eq.~\eqref{II 6}; $\two_{n,7}^\mm$ in Eq.~\eqref{II n,7 estimate}; $\two_{n,8}^\mm$ in Eqs.~\eqref{II 8, n odd} and \eqref{II 8, n even};  as well as  $\two_{n,9}^\mm$ in Eqs.~\eqref{II 9}. Putting together all these estimates, we arrive that
\begin{align*}
\Phi_n^\mm(t,p) =  \Phi_{n, {\rm good}}^\mm(t,p)  + {\rm Error}_n^\mm(t,p)
\end{align*}
with $\left\| {\rm Error}_n^\mm(t,p)\right\| \leq C\left(K, \Lambda, \lambda^{-1}, n,d,t \right)\cdot \left(1+|p|^{n-2}\right)\cdot m$. It only remains to check that 
\begin{equation}\label{June25, x}
    \Phi_{n, {\rm good}}^\mm(t,p) = \fraka_n^\mm(t,p) + \frakb_{n}^\mm(t,p) + E'(t,p)\qquad \text{for any $n \geq 3$}.
\end{equation} Throughout the rest of the proof, we label any term satisfying $\left\|E'\right\| \leq C\left(K, \Lambda, \lambda^{-1}, n,d,t \right)\cdot \left(1+|p|^{n-2}\right)\cdot m$ as $E'$, which may change from line to line.

To deduce Eq.~\eqref{June25, x} it suffices to check that, for any even number $n \geq 4$,
\begin{align*}
\left[\Xi^\mm(t)\right]^{\otimes \frac{n}{2}} t^{\frac{n}{2}} =  K_1 + K_2 + K_3 +E',    
\end{align*}
where
\begin{eqnarray*}
 &&K_1 =  \scrm \cdot \left(\dashint_{0}^{\frac{t}{m}} \ssig_\sigma \,\dd \sigma \right) \otimes \left[\Xi^\mm(t)\right]^{\otimes \left(\frac{n}{2}-1\right)} t^{\frac{n}{2}},\\
 && K_2 = \sum_{j=1}^d e_j \otimes \Bigg[ \int_0^t \left[{-}\id {+} e^{-\mom (t-s)}\right]^{j}_{\bullet}\otimes \left[\Xi^\mm(t-s)\right]^{\otimes \left(\frac{n}{2}-1\right)} (t-s)^{\frac{n}{2}-1}\,\dd s\Bigg],\\
 &&K_3 =  \frac{\id}{2}\otimes  \left[\int_0^t \left[\Xi^\mm(t-s)\right]^{\otimes  \left(\frac{n}{2}-1\right)}(t-s)^{\left(\frac{n}{2}-1\right)} \,\dd s\right].
\end{eqnarray*}

For this purpose, recall the definition of $\Xi^\mm$ from Eq.~\eqref{Xi, def} and notice first that 
\begin{align}\label{K1, endgame}
K_1 = \left(t\Xi^\mm(t) + \frac{\id}{2}t\right) \otimes \left[\Xi^\mm(t)\right]^{\otimes \left(\frac{n}{2}-1\right)} t^{\frac{n}{2}-1}.
\end{align}

Next, we split $K_2$ into two terms:
\begin{align*}
    K_2  &= K_{2,1} + K_{2,2}:= -\id \otimes \left[\int_0^t\left[\Xi^\mm(t-s)\right]^{\otimes \left(\frac{n}{2}-1\right)} (t-s)^{\frac{n}{2}-1}\,\dd s\right] \\
    &\qquad + \sum_{j=1}^d e_j \otimes \left[ \int_0^t \left(e^{-\mom (t-s)}\right)^{j}_{\bullet} \otimes \left[\Xi^\mm(t-s)\right]^{\otimes \left(\frac{n}{2}-1\right)} (t-s)^{\frac{n}{2}-1} \,\dd s \right].
\end{align*}

For the first term,  we shift to $\sigma := s \slash m$ to express
\begin{align*}
    K_{2,1} &= -\id \otimes \left\{m \int_0^{\frac{t}{m}} \left[ \Xi^\mm(t-m\sigma)^{\otimes \left(\frac{n}{2} - 1 \right)} (t-m\sigma)^{\frac{n}{2} - 1} \right]\,\dd \sigma\right\}.
\end{align*}
From Lemma~\ref{lem: comparing shift in time, Xi(t-m.)} we deduce that
\begin{align*}
&\left\|\left[\Xi^\mm(t-m\tau_0 )\right]^{\otimes \left(\frac{n}{2} - 1 \right)} (t-m\tau_0 )^{ \frac{n}{2} - 1 } - \left[\Xi^\mm(t )\right]^{\otimes \left(\frac{n}{2} - 1 \right)}t^{\frac{n}{2} - 1}\right\| \leq C\left(K,\Lambda,\lambda^{-1},n,d,t\right)\cdot m.
\end{align*}
Thus, $\left\| K_{2,1} + \id \otimes \left[\Xi^\mm(t )\right]^{\otimes \left(\frac{n}{2} - 1 \right)}t^{\frac{n}{2}} \right\| \leq C\left(K,\Lambda,\lambda^{-1},n,d,t\right)\cdot m$.

The estimate for $K_{2,2}$ is simple due to the presence of the exponential decaying factor. Indeed, by considering $\tau := \frac{t-s}{m}$, one obtains
\begin{align*}
    \left\|K_{2,2}\right\| &\leq m\cdot K\int_0^{\frac{t}{m}} e^{-\lambda\tau} \left\|\Xi^\mm(m\tau)\right\|^{\frac{n}{2}-1} (m\tau)^{\frac{n}{2}-1}\,\dd\tau.
\end{align*}
Here, one may estimate $\|\Xi^\mm(\bullet)\|$ uniformly as in Eq.~\eqref{Xi m, uniform bd}, and the function $\tau \mapsto e^{-\lambda \tau} \tau^{\frac{n}{2}-1}$ vanishes as both $\tau \to 0^+$ and $\tau \to \infty$, hence integrable on $\R_+$. We thus have 
\begin{align*}
\left\|K_{2,2}\right\| &\leq C\left(K,\lambda^{-1},n,d\right) m^{\frac{n}{2}} \int_0^{\frac{t}{m}} e^{-\lambda \tau} \tau^{\frac{n}{2}-1}\,\dd \tau\leq C\left(K,\lambda^{-1},n,d,t\right) m^{\frac{n}{2}-1}\qquad\text{for even $n \in \mathbb{Z}_{\geq 4}$}.
\end{align*}

In summary, $K_2$ has been estimated by
\begin{align}\label{K2, endgame}
\left\|K_2 +\id\otimes \left[\Xi^\mm(t )\right]^{\otimes \left(\frac{n}{2} - 1 \right)}t^{\frac{n}{2}}\right\| \leq C\left(K,\Lambda,\lambda^{-1},t,n,d\right)\cdot m
\end{align}
 whenever $m \in ]0,1]$ and $n \in \mathbb{Z}_{\geq 4}$ is even.

Finally, for $K_3$, invoking Lemma~\ref{lem: comparing shift in time, Xi(t-m.)} once again, we find that
\begin{align}\label{K3, endgame}
    K_3 =  \left\{\left(\frac{\id}{2} t\right) \otimes \left[\Xi^\mm(t)\right]^{\otimes \left( \frac{n}{2}-1\right)} t^{\frac{n}{2}-1}\right\} + E'. 
\end{align}

Now, putting together Eqs.~\eqref{K1, endgame}, \eqref{K2, endgame}, and \eqref{K3, endgame} (noticing in particular the cancellations between all the terms that are constant multiples of $\id \otimes \left[\Xi^{\mm}\right]^{\frac{n}{2}-1} t^{\frac{n}{2}}$), we arrive at
\begin{align*}
K_1+K_2+K_3 &= t\Xi^\mm (t) \otimes \left[\Xi^{\mm}\right]^{\otimes \left(\frac{n}{2}-1\right)} t^{\frac{n}{2}-1} + E' \equiv \left[\Xi^{\mm}\right]^{\otimes \frac{n}{2}} t^{\frac{n}{2}} +E'. 
\end{align*}
This proves Eq.~\eqref{June25, x}.   
The proof of Theorem~\ref{main theorem for singular limit} is now complete.   \end{proof}

\section{Explicit expression for the small mass limit for diagonalisable $\scrm$}\label{sec: diag}

In this section, we further investigate the structure of the term 
\begin{align*}
\overline{\fraka}_{q}(p) := \idotsint\limits_{0\leq t_1<\dots<t_q < \infty} \bigotimes_{j=1}^q \left(-\scrm e^{-\scrm t_j}p\right) \,\dd t_1\ldots\,\dd t_q\quad (q\in\mathbb{N}),
\end{align*}
which occurs in the zero-mass limit of the expected signature $\Phi_n^\mm(t,p)$ of the momentum path. Indeed, $    \overline{\fraka}_{q}(p) = \lim_{m \to 0^+} \fraka_q^\mm(t,p)$ for any $t>0$, so by Corollary~\ref{cor, explicit limit} we have  
\begin{equation*}
    \lim_{m \to 0^+} \Phi^\mm_n(t,p) = \sum_{k=0}^{\lfloor {n}\slash{2}\rfloor} \left\{\overline{\fraka}_{n-2k}(p) \otimes  \left(\scrm \scrc -\frac{\id}{2} \right)^{\otimes k}  \right\}t^k.
\end{equation*}
When the positive definite matrix $\scrm$ is diagonalisable, $\overline{\fraka}_{q}(p)$ can be explicitly evaluated, with the resulting expression manifesting interesting combinatorial features. See Example~(Diagonalisable $\scrm$) below. Also note that the small mass limit $ \lim_{m \to 0^+} \Phi^\mm_n(t,p)=0$ if and only if $\scrm$ is symmetric.

We start from the following simple observation. For a finite $t>0$, set \begin{align}\label{A, new}
    \mathcal{A}^{c_1, c_2, \dots, c_n}(t) := \idotsint\limits_{0\leq t_1<\dots<t_n\leq t}e^{-c_1 t_1}\,\dd t_1 \dots e^{-c_n t_n}\,\dd t_n,
\end{align}
where $c_1, \ldots, c_n$ are complex numbers of positive real parts. 

\begin{lemma}\label{lem: iterated integral}
$\mathcal{A}^{c_1, c_2, \dots, c_n}$ defined in Eq.~\eqref{A, new} can be evaluated as follows: 
\begin{align*}
    \mathcal{A}^{c_1, c_2, \ldots, c_n}(t) &=A^n_0+A^n_1 e^{-c_nt}+A^n_2  e^{-(c_n+c_{n-1})t} +\ldots+A^n_n e^{-(c_n+\ldots+c_1)t},
\end{align*}
with coefficients $A^n_i$ given by the recursive relation:
\begin{align*}
    \left\{
        \begin{array}{ll}
        A^{n}_i = -\frac{A^{n-1}_{i-1}}{c_{n}+\cdots+c_{n-i+1}} \text{ for each $i \in \{1,\ldots,n\}$},\qquad A^{n}_0 = -\sum_{i = 1}^
        {n}A^{n}_i \prod_{i=1}^n\frac{1}{\sum_{j=i}^n c_j}; 
        \\ A^1_0 = \frac{1}{c_1}, \quad A^1_1 = -\frac{1}{c_1}. 
        \end{array}
    \right.
\end{align*}

\end{lemma}

\begin{proof}[Proof for Lemma~\ref{lem: iterated integral}]
Induct on $n$. For $n = 1$, $A^{c_1}(t) = \int_0^t e^{-c_1 t_1}\, \dd t_1 = -\frac{1}{c_1}e^{-c_1 t}+\frac{1}{c_1}$, so  $A^1_0 = \frac{1}{c_1}=-A^1_1$. Now assume the expression for $n$. Then
\begin{align*}
\mathcal{A}^{c_1, \dots, c_{n+1}}(t) &=\int_{0\leq t_1<\cdots<t_{n+1}\leq t}e^{-c_1 t_1}\,\dd t_1 \cdots e^{-c_n t_n}\,\dd t_n\; e^{-c_{n+1} t_{n+1}} \,\dd t_{n+1}\\
&=\int_0^t \mathcal{A}^{c_1, \dots, c_{n}}\left(\tilde{t}_1\right)\cdot e^{-c_{n+1}\tilde{t}_1}\,\dd \tilde{t}_1\\
    &= \int_0^t\left\{ \sum_{i = 1}^{n} A^n_i e^{-(c_n+\cdots+c_{n-i+1}) \tilde{t}_1} + A^n_0\right\}\cdot e^{-c_{n+1}\tilde{t}_1}\,\dd \tilde{t}_1\\
    &=\sum_{i = 0}^{n}\int_0^t A^n_{i} e^{-(c_{n+1}+\cdots+c_{n-i+1}) \tilde{t}_1}\,\dd \tilde{t}_1=\sum_{i=0}^{n}\frac{A^n_{i}\left[1-e^{-(c_{n+1}+\cdots+c_{n-i+1}) t}\right]}{c_{n+1}+\cdots+c_{n-i+1}},
\end{align*}
where in the third line we used the inductive hypothesis. This implies that 
\begin{align*}
    A^{n+1}_i = -\frac{A^{n}_{i-1}}{c_{n+1}+\cdots+c_{n-i+2}} \qquad \text{for } i \in \{1, \ldots, n+1\}.
\end{align*}
By taking $t=0$ we deduce $ \mathcal{A}^{c_1, \dots, c_{n+1}}(0)=\sum_{i=0}^{n+1}A^{n+1}_i = 0$. This completes the induction.

To determine $A^{n}_0$, note that it equals to $\mathcal{A}^{c_1, c_2, \dots, c_n}(\infty)$ by Fubini's Theorem: 
\begin{align*}
    \mathcal{A}^{c_1, c_2, \dots, c_n}(\infty) &=\idotsint\limits_{0\leq t_1<\dots<t_n<\infty}e^{-c_n t_n}\,\dd t_n\dots e^{-c_1 t_1}\,\dd t_1 =\idotsint\limits_{0\leq t_1<\dots<t_{n-1}<\infty}\frac{1}{c_n}e^{-(c_{n-1}+c_n) t_{n-1}}\,\dd t_{n-1}\dots e^{-c_1 t_1}\,\dd t_1\\
    &=\idotsint\limits_{0\leq t_1<\dots<t_{n-2}<\infty}\frac{e^{-(c_{n}+c_{n-1}+c_{n-2}) t_{n-2}}\,\dd t_{n-2}\dots e^{-c_1 t_1}\,\dd t_1}{c_n(c_n+c_{n-1})}=\prod_{i=1}^n\frac{1}{c_i+\cdots+c_n}.
\end{align*}
Indeed, all the other modes $\mathcal{A}^{c_1, c_2, \dots, c_n}(t)$ decay exponentially as $t \to \infty$. \end{proof}

Lemma~\ref{lem: iterated integral} above enables us to find explicit expressions for 
 $\overline{\fraka}_n(p)$ when $\scrm$ is diagonalisable.

\begin{example}[Isotropic diagonal $\scrm$]\label{example, identity}
If $\scrm = \lambda\cdot \id$ for  $\lambda \in \{\mathfrak{Re}> 0\}$, then $\overline{\fraka}_n(p) = \frac{1}{n!}(-p)^{\otimes n}$ for any $p \in \R^d$, $n \in \mathbb{N}$.
\end{example}
\begin{proof}
The integrand in $\overline{\fraka}_n(p)$ is symmetric in each $t_i$. Thus, by Lemma~\ref{lem: sym},
\begin{align*}
\overline{\fraka}_{n}(p) 
&= \idotsint\limits_{0\leq t_1<\dots<t_n < \infty} \bigotimes_{j=1}^n \left(-\lambda\cdot\id\; e^{-t_j\lambda\cdot\id}p\right) \,\dd t_1\ldots\,\dd t_n =\frac{1}{n!}\bigotimes_{j=1}^n\left(\int_0^{\infty}-\lambda e^{-\lambda t_j}p \,\dd t_j\right)=\frac{1}{n!}(-p)^{\otimes n}.
\end{align*} 
\end{proof}

When $\scrm$ is diagonal with distinctive eigenvalues, the trick of symmetrisation is no longer applicable, but we can use instead the formula for $A^n_0$ in Lemma~\ref{lem: iterated integral} to evaluate $\overline{\fraka}_{n}(p)$.

\begin{example}[General diagonal $\scrm$]\label{example: M diagonal}
Suppose $\scrm = \scrd = {\rm Diag}\,(\lambda_1, \cdots,\lambda_d)$ with $\mathfrak{Re}(\lambda_i)>0$ for each $i$. Then for any $p \in \R^d$, $n \in \mathbb{N}$, and $I = \left(i_1, \ldots, i_n\right) \in \{1,\ldots,d\}^n$, we have $
    \left[\overline{\fraka}_{n}(p)\right]^I = \prod_{j=1}^n\frac{\lambda_{i_j}p_{i_j}}{\lambda_{i_j}+\cdots+\lambda_{i_n}}.$
\end{example}

\begin{proof}  Since both $\scrd$ and $e^{-\scrd t_j}$ are diagonal $d \times d$ matrices, we have 
\begin{align*}
    \left[\overline{\fraka}_{n}(p)\right]^I
    &= \left[\;\;\idotsint\limits_{0\leq t_1<\dots<t_n < \infty} \bigotimes_{j=1}^n \left(-\scrd e^{-\scrd t_j}p\right) \,\dd t_1\ldots\,\dd t_n\right]^I\\
    &=\idotsint\limits_{0\leq t_1<\dots<t_n < \infty}\left[\prod_{j=1}^n \left(-\scrd e^{-\scrd t_j}p \right)^{i_j} \right]\dd t_1\ldots\dd t_n =\idotsint\limits_{0\leq t_1<\dots<t_n < \infty}\prod_{j=1}^n \left(-\lambda_{i_j} e^{-\lambda_{i_j} t_j} p^{i_j} \right)\,\dd t_1\ldots\,\dd t_n\\
    &=\prod_{j=1}^n\left(-\lambda_{i_j}p^{i_j}\right)\cdot \idotsint\limits_{0\leq t_1<\dots<t_n < \infty}\,\,\prod_{j=1}^n \left(e^{-\lambda_{i_j} t_j} \right)\dd t_1\ldots\dd t_n=\prod_{j=1}^n\left(\frac{-\lambda_{i_j}p^{i_j}}{\lambda_{i_j}+\cdots+\lambda_{i_n}}\right).
\end{align*}
The last line holds by the expression  for $A^n_0$ in Lemma~\ref{lem: iterated integral} with $c_i = \lambda_i$.   \end{proof}

Finally, when $\scrm$ is diagonalisable, we compute $\overline{\fraka}_{n}(p)$ via the mixed product property of tensor products and the argument for  Example~(General diagonal $\scrm$). This completes the proof of  Theorem~\ref{thm: main, Dec24}.

\begin{example}[Diagonalisable $\scrm$]\label{example: main}
Suppose $\scrm \in \mathfrak{gl}(d;\C)$ is diagonalisable with $\scrm = Q\scrd Q^{-1}$ for $\scrd= {\rm Diag}\,(\lambda_1, \cdots,\lambda_d)$;  $\mathfrak{Re}(\lambda_j)>0$ for each $j$. Then for any $p \in \R^d$, $n \in \mathbb{N}$, and multi-index $I = \left(i_1, \ldots, i_n\right) \in \{1,\ldots,d\}^n$, it holds that 
\begin{align}\label{diagonalisable, new}
   \left[ \overline{\fraka}_{n}(p)\right]^I =  \sum\Bigg\{Q^{i_1}_{j_1} \cdots Q^{i_n}_{j_n} \mathcal{Z}_{j_1,\ldots,j_n}\left(Q^{-1}\right)^{j_1}_{r_1}\cdots\left(Q^{-1}\right)^{j_n}_{r_n} p^{r_1}\cdots p^{r_n}\Bigg\}, 
\end{align}
where $\mathcal{Z}_{j_1, \ldots, j_n}:=\prod_{\ell=1}^n \left( \frac{-\lambda_{j_\ell}}{\lambda_{j_{\ell}} + \cdots + \lambda_{j_n}}\right)$ and the summation is over $j_1, \ldots, j_n; r_1, \ldots, r_n \in \{1,\ldots, d\}$.
\end{example}

This above result reduces to Example~(General diagonal $\scrm$) when $Q= \id$.

\begin{proof}
By the mixed product property in \S\ref{Tensor product of linear mappings} and the simultaneous diagonalisability of $\scrm$ and $e^{-\scrm t_j}$, we have that 
\begin{align*}
\overline{\mathfrak{A}}_n(p) = Q^{\otimes n}  \Lambda  Q^{\otimes (-n)} p^{\otimes n},\qquad\Lambda:= \idotsint\limits_{0\leq t_1<\dots<t_n < \infty} \bigotimes_{j=1}^n \left(-\scrd e^{-\scrd t_j}\right)\,\dd t_1\ldots\dd t_n.
\end{align*} 
Here $Q^{\otimes n}$, $\Lambda$, and $Q^{\otimes (-n)} := \left(Q^{-1}\right)^{\otimes n} \in \mathfrak{gl}(d;\C)^{\otimes n}$. See \S\ref{Tensor product of linear mappings} for notation. 

It remains to specify $\Lambda = \left\{\Lambda^I_J\right\}$ for arbitrary multi-indices $I=(i_1, \ldots, i_n)$ and $J = (j_1, \ldots, j_n) \in \{1,\ldots,d\}^n$. We have $-\scrd e^{-\scrd t_j} = {\rm Diag}\,\left(-\lambda_1 e^{-\lambda_1 t_j}, \ldots, -\lambda_n e^{-\lambda_nt_j}\right)$ for $\scrd = {\rm Diag}\,(\lambda_1, \ldots, \lambda_n)$. Thus,\\
    $\Lambda^I_J = \prod_{\ell=1}^n\,\, \left\{\,\, \idotsint\limits_{0\leq t_1<\dots<t_n < \infty}  \left(-\lambda_{i_\ell} e^{-\lambda_{i_\ell} t_\ell} \delta_{i_\ell, j_\ell}\right)\,\dd t_1\ldots\dd t_n\right\},$
where $\delta_{i_\ell, j_\ell}$ is the Kronecker delta. Computations in Example~(General diagonal $\scrm$) (with $p_{i_j}=1$) yield that
\begin{align*}
    \Lambda^I_J = \prod_{\ell=1}^n \left( \frac{-\lambda_{i_\ell}}{\lambda_{i_{\ell}} + \cdots + \lambda_{i_n}}\right)\,\delta_{i_\ell, j_\ell}.
\end{align*}

As a consequence, for each multi-index $K \in \{1,\ldots,d\}^n$ we have
\begin{align*}
\left[\overline{\fraka_n}(p)\right]^K &= \sum_{I,J,R\in \{1,\ldots,d\}^n} \left(Q^{\otimes n}\right)^K_I \Lambda^I_J \left(Q^{\otimes -n}\right)^J_R \left(p^{\otimes n}\right)^R\\
&= \sum_{i_1, \ldots, i_n; j_1, \ldots, j_n; r_1, \ldots, r_n \in \{1,\ldots, d\}} \Bigg\{Q^{k_1}_{i_1} \cdots Q^{k_n}_{i_n} \cdot \mathcal{Z}_{i_1, \ldots, i_n} \cdot \delta^{i_1}_{j_1}\cdots\delta^{i_n}_{j_n}\cdot\left(Q^{-1}\right)^{j_1}_{r_1}\cdots\left(Q^{-1}\right)^{j_n}_{r_n}\cdot p^{r_1}\cdots p^{r_n}\Bigg\}\\
&= \sum_{i_1, \ldots, i_n; r_1, \ldots, r_n \in \{1,\ldots, d\}} \Bigg\{Q^{k_1}_{i_1} \cdots Q^{k_n}_{i_n} \mathcal{Z}_{i_1,\ldots,i_n}\left(Q^{-1}\right)^{i_1}_{r_1}\cdots\left(Q^{-1}\right)^{i_n}_{r_n}\cdot p^{r_1}\cdots p^{r_n}\Bigg\}, 
\end{align*}
where $\mathcal{Z}_{i_1, \ldots, i_n}:=\prod_{\ell=1}^n \left( \frac{-\lambda_{i_\ell}}{\lambda_{i_{\ell}} + \cdots + \lambda_{i_n}}\right)$. This completes the proof.  \end{proof}

\begin{remark}
Example~(Diagonalisable $\scrm$), Eq.~\eqref{diagonalisable, new} cannot be simplified in general, as $\mathcal{Z}_{j_1, \ldots, j_n}$ is nonlinear in each coordinate. 
To illustrate the non-(multi-)linearity of the expression~\eqref{diagonalisable, new} for $\overline{\fraka}_n(p)$, we remark that even for $d=n=2$, one cannot find $2\times 2$ matrices $M_1$ and $M_2$ such that $\Lambda = M_1 \otimes M_2$, unless $\scrd = {\rm Diag}\,(\lambda, \lambda)$; \emph{i.e.}, when it is reduced to Example~(Isotropic diagonal $\scrm$).

Indeed, for $\scrd = {\rm Diag}\,(\lambda_1, \lambda_2)$ we compute that
\begin{align*}
    \Lambda &= {\rm Diag}\, \Bigg(\iint_{\blacktriangle}(\lambda_1)^2 e^{-\lambda_1(t_1+t_2)}\,\dd t_1\,\dd t_2,\, \iint_{\blacktriangle}\lambda_1 \lambda_2 e^{-\lambda_1 t_1- \lambda_2 t_2}\,\dd t_1\,\dd t_2, \\
    &\qquad \iint_{\blacktriangle}\lambda_1 \lambda_2 e^{-\lambda_2 t_1- \lambda_1 t_2}\,\dd t_1\,\dd t_2, \, \iint_{\blacktriangle}(\lambda_2)^2 e^{-\lambda_2(t_1+t_2)}\,\dd t_1\,\dd t_2\Bigg),
\end{align*}
where $\blacktriangle := \{(t_1,t_2):\,0\leq t_1<t_2<\infty\}$ and the tensor product is identified with the Kronecker product of matrices, \emph{i.e.}, via $\mathfrak{gl}(2;\C) \otimes \mathfrak{gl}(2;\C) \cong \mathfrak{gl}(4;\C)$. Using the trick of symmetrisation in Lemma~\ref{lem: sym} and the computations in Example~(General diagonal $\scrm$), we obtain $\Lambda = {\rm Diag}\,\left(\frac{1}{2},\,\frac{\lambda_1}{\lambda_1+\lambda_2},\,\frac{1}{2},\,\frac{\lambda_2}{\lambda_1+\lambda_2}  \right)$.

Suppose that there are $2\times 2$ matrices $M_1$ and $M_2$ with $\Lambda = M_1 \otimes M_2$. Then $M_1$ and $M_2$ are diagonal; say, $M_1 = {\rm Diag}\,(M_{1,1}, M_{1,2})$ and $M_2 = {\rm Diag}\,(M_{2,1}, M_{2,2})$. But in this case $$M_1 \otimes M_2 = {\rm Diag}\,\left(M_{1,1}M_{2,1}, M_{1,1}M_{2,2},M_{1,2}M_{2,1},M_{1,2}M_{2,2}\right),$$ which cannot be equal to $\Lambda$ unless $\lambda_1 = \lambda_2$.   
\end{remark}

\begin{appendix}
\section*{Proof of Technical lemmas}
\label{sec: appendix, lemma}

The proofs of several technical lemmas in the main text are presented in Appendix~\ref{sec: appendix, lemma}.




\begin{proof}[Proof of Lemma~\ref{lem: matrix expoential limit}]
Let $J$ the Jordan canonical form of $\scrm$. Then $\scrm = PJP^{-1}$ and hence $e^{-\scrm} = P e^{-J} P^{-1}$. It suffices to prove for the case when $J$ has only one Jordan block and $\delta=1$. Consider the Jordan decomposition $J = \lambda\, \id + N$ for eigenvalue $\lambda$ and nilpotent $N$. Compute via the binomial theorem for commuting matrices:
\begin{align*}
    e^{-J} &=\sum_{k=0}^{\infty}\frac{1}{k!}(-\lambda \,\id -N)^k=\sum_{k=0}^{\infty}\frac{1}{k!}\sum_{j=0}^{k} {k \choose j} (-\lambda)^{k-j}(-N)^j= e^{-\lambda} \left(\id-N+\frac{N^2}{2!}- \cdots +(-1)^{d-1}\frac{N^{d-1}}{(d-1)!}\right).
\end{align*}
Thus, $\left\|e^{-\scrm}\right\| \leq \left\|e^{N}\right\|\left\|P\right\|\left\|P^{-1}\right\| e^{-\lambda}$. A scaling argument yields $\left\|e^{-\scrm\delta}\right\| \leq K e^{-\lambda \delta}$ for any $\delta>0$.  \end{proof}

\begin{proof}[Proof of Lemma~\ref{lem: average difference}]
Direct calculation gives us $
  \dashint_0^t f - \dashint_0^{t-\delta} f = 
\frac{1}{t}\int_{t-\delta}^t f -  \frac{\delta}{t(t-\delta)}\int_0^{t-\delta}f$. The absolute value of each term on the right-hand side is $\leq C_0\delta t^{-1}$.  \end{proof}

\begin{proof}[Proof of Lemma~\ref{lem: sym}]
For each permutation of $n$ letters $\sigma$, write
$\blacktriangle^\sigma := \left\{ \left(t_1, \ldots, t_n\right) \in \R^n:\right.$ $\left. 0 \leq t_{\sigma(1)}< \ldots < t_{\sigma(n)} \leq L\right\}.$ The symmetry of $F$ implies $\int_{\blacktriangle^\sigma}F = \int_{\blacktriangle}F$. Hence, $    \int_{\blacktriangle}F = \frac{1}{\# \Sigma_n}$ $ \int_{\bigcup_{\sigma \in \Sigma_n} \blacktriangle^\sigma} F = \frac{1}{n!} \int_{[0,L]^n} F = \frac{1}{n!} \left[\int_0^L f(t)\,\dd t\right]^{\otimes n}$.
 \end{proof}

\begin{proof}[Proof of Lemma~\ref{lem: pure expectation is O(m)}] In view of Lemma~\ref{lem:GaussianMoment}, we have
\begin{align*}
\left\|\E^p \left[P_s^{\otimes \ell}\right]\right\| &\leq \int_0^t \Bigg\|\left(e^{-\mom s}p\right)^{\otimes \ell} + \sum_{j=1}^{\lfloor \ell\slash 2\rfloor} C_{j,\ell}\, \sym\Bigg[\left(e^{-\mom s}p\right)^{\otimes( \ell-2j)} \otimes\left(\int_0^s e^{-\mom \varsigma} e^{-\frac{\scrm^*}{m}\varsigma}\,\dd \varsigma\right)^{\otimes j} 
\Bigg]\Bigg\|\,\dd s\\
&\leq m\int_0^{\frac{t}{m}} \Bigg\{ \left\|e^{-\scrm \sigma }p\right\|^\ell +  \sum_{j=1}^{\lfloor \ell\slash 2\rfloor} C_{j,\ell}\, m^j\, \left\|e^{-\scrm \sigma }p\right\|^{\ell - 2j} \left\|\ssig_{\sigma}\right\|^j
\Bigg\}\,\dd \sigma.
\end{align*}
The second line follows from the change of variables $\sigma:=s/m$. Clearly, $\int_0^\infty \left\|e^{-\scrm \sigma}\right\|\,\dd \sigma \leq K\lambda^{-1}$, so the integral is of order $\mathcal{O}(1)$.  \end{proof}

\begin{proof}[Proof of Lemma~\ref{lem for II 7}]
Thanks to the definition of $\fraka_\ell^\mm$ and the fact that the term  $\bigotimes_{j=1}^\ell  \left( -\scrm e^{-\scrm t_j} \right)$ is deterministic (\emph{i.e.}, it is constant with respect to $\E^p$), we deduce that 
\begin{align*}
&\E^p\Bigg[\int_0^t \left\|\fraka_\ell^\mm(t-s,P_s)\right\| \,\dd s \Bigg] \leq \int_0^td^{\ell}\left\| \,\,\,\,\idotsint\limits_{0\leq t_1<\dots<t_\ell\leq \frac{t-s}{m}} \,\,\bigotimes_{j=1}^\ell  \left( -\scrm e^{-\scrm t_j} \right)\, \dd t_1\ldots\,\dd t_\ell
\right\| \cdot \left\|\E^p\left[P_s^{\otimes \ell}\right]\right\| \,\dd s.
\end{align*}
The tensor norm of the iterated integral in the parenthesis is estimated by 
\begin{align*}
\left\|\,\,\,\idotsint\limits_{0\leq t_1<\dots<t_\ell\leq \frac{t-s}{m}} \,\,\bigotimes_{j=1}^\ell  \left( -\scrm e^{-\scrm t_j} \right)\, \dd t_1\ldots\,\dd t_\ell
\right\| \leq \Lambda^\ell K^{\ell}\,\,\idotsint\limits_{0\leq t_1<\dots<t_\ell\leq \frac{t-s}{m}} \,\, e^{-\lambda(t_1+\ldots+ t_\ell)}\,  \dd t_1\ldots\,\dd t_\ell \leq \left(\frac{\Lambda K}{\lambda}\right)^\ell \frac{1}{\ell!},
\end{align*}
thanks to the symmetrisation trick in Lemma~\ref{lem: sym}. In particular, this bound is uniform in $s$ and $m$. Thus
\begin{align*}    \E^p\Bigg[\int_0^t \left\|\fraka_\ell^\mm(t-s, P_s)\right\| \,\dd s \Bigg] \leq \left(\frac{\Lambda K d}{\lambda}\right)^\ell \frac{1}{\ell!} \int_0^t \left\|\E^p\left[P_s^{\otimes \ell}\right]\right\| \,\dd s.
\end{align*}
We now conclude the proof from  Lemma~\ref{lem: pure expectation is O(m)}.  
\end{proof}

\begin{proof}[Proof of Lemma~\ref{lem: handling B in II 2,5,8}]
By Eq.~\eqref{Xi estimate}
 and the elementary inequality $|a+b|^r \leq 2^{r-1}\left[|a|^r+|b|^r\right]$ for $r \geq 0$, we have
\begin{align*}
\left\|\left[\Xi^\mm(t-s)\right]^{\otimes r} (t-s)^r \right\| &\leq 2^{r-1} t^r \left\{\left\|\Xi^\mm(t) \right\|^r +  \left(\frac{2\Lambda K^2 d^2}{\lambda}\right)^r\right\}. \end{align*}
This together with Eq.~\eqref{Xi m, uniform bd} for $\left\|\Xi^\mm(t) \right\|$ gives us
\begin{align*}
\left\|\E^p\left[ \int_0^t \fraka^\mm_\ell(t-s, P_s) \otimes \left(\Xi^\mm(t-s)\right)^{\otimes r} (t-s)^r \,\dd s\right]\right\| \leq C\left(K,\Lambda,\lambda^{-1},r,d,t\right) \E^p\left[\int_0^t \left\|\fraka^\mm_\ell(t-s, P_s)\right\| \,\dd s \right].
\end{align*}
The assertion now follows immediately from Lemma~\ref{lem for II 7}.  \end{proof}

\begin{proof}[Proof of Lemma~\ref{lem: comparing shift in time, Xi(t-m.)}]
We bound by triangle inequality
\begin{align*}
    &\left\|\left[\Xi^\mm(t-m\tau_0 )\right]^{\otimes j} (t-m\tau_0 )^j - \left[\Xi^\mm(t )\right]^{\otimes j} t^j\right\|\\
    &\quad \leq \left\|\left\{\left[\Xi^\mm(t-m\tau_0 )\right]^{\otimes j} - \left[\Xi^\mm(t )\right]^{\otimes j}  \right\} \left(t-m\tau_0\right)^j\right\|+ \left\|\left[\Xi^\mm(t )\right]^{\otimes j} \left\{t^j-(t-m\tau_0 )^j\right\}\right\| =: \Xi_1 + \Xi_2.
\end{align*}

For $\Xi_1$,  arbitrary tensors $A$ and $B$ of the same type satisfy
\begin{align*}
A^{\otimes j} - B^{\otimes j} &= \sum_{\{q_1,q_2\in\mathbb{N}:\,q_1+q_2=j-1\}} B^{\otimes q_1} \otimes (A-B) \otimes A^{\otimes q_2}.
\end{align*}
Specialising to $A= \Xi^\mm(t-m\tau_0)$, $B=\Xi^\mm(t)$, we have $\|A\|$ and $\|B\| \leq \frac{1}{2}\left(1+ \frac{\Lambda K^2d^2}{\lambda}\right)\leq \frac{\Lambda K^2d^2}{\lambda}$. See Eq.~\eqref{Xi m, uniform bd}. On the other hand, by definition for $\Xi^\mm$, the bound~\eqref{Sigma estimate, simple} for $\ssig$, as well as Lemma~\ref{lem: average difference}, we have
 \begin{align*}
 \|A-B\| \leq d\|\scrm\| \left\| \left\{ \dashint_0^{\frac{t-m\tau_0}{m}} - \dashint_0^{\frac{t}{m}}   \right\}\ssig_\varsigma\,\dd\varsigma \right\|\leq \frac{2\Lambda K^2d^2 \tau_0}{\lambda t}\cdot m \leq \frac{2\Lambda K^2d^2}{\lambda}\cdot m. 
 \end{align*}
It thus follows that
 $   \Xi_1 \leq t^j \left(\frac{\Lambda K^2d^2}{\lambda}\right)^{j-1}\frac{\Lambda K^2d^2}{\lambda}\cdot m = C\left(K,\Lambda, \lambda^{-1},j,d,t\right)\cdot m.$

For $\Xi_2$, as above $\left\|\left[\Xi^\mm(t)\right]^{\otimes j}\right\| \leq \left(\frac{\Lambda K^2d^2}{\lambda}\right)^{j}$, and $
    t^j- (t-m\tau_0)^j = \sum_{q=1}^{j} {j \choose q} (-m\tau_0)^q t^{j-q} \leq C(j,t) \cdot m$ for $m>0$ small, thanks to the binomial theorem.  \end{proof}

\end{appendix}
\subsection*{Acknowledgments}
HN thanks Terry Lyons for useful discussions. The authors also thank the anonymous referees for their constructive comments
that improved the quality of this paper.

\subsection*{Funding}
SL and HN are both supported by the SJTU-UCL joint seed fund WH610160507/067.  SL is also supported by NSFC Projects 12201399, 12331008, and 12411530065, Young Elite Scientists Sponsorship Program by CAST 2023QNRC001, the National Key Research $\&$ Development Programs 2023YFA1010900 and 2024YFA1014900, Shanghai Rising-Star Program 24QA2703600, and the Shanghai Frontier Research Institute for Modern Analysis. HN is also supported by the EPSRC under the program grant EP/S026347/1 and the Alan Turing Institute under the EPSRC grant EP/N510129/1.

\end{document}